\documentclass[12pt]{amsart}
\usepackage{amsmath}
\usepackage{amsxtra}
\usepackage{amscd}
\usepackage{amsthm}
\usepackage{amsfonts}
\usepackage{amssymb}
\usepackage{eucal}
\usepackage{epsfig}
\usepackage{graphics}
\usepackage{graphicx}
\usepackage{accents}
\usepackage{dynkin-diagrams}
\usepackage{comment}
\usepackage{enumitem}

\usepackage{mathtools}
\usepackage{tikz-cd}
\usepackage{mathtools}
\usepackage{pgf}
\usetikzlibrary{arrows, matrix}

\usepackage{here} 
\usepackage{simplewick}
\numberwithin{equation}{section}
\newtheorem{thm}{Theorem}[section]

\newtheorem{lem}[thm]{Lemma}
\newtheorem{rem}[thm]{Remark}

\newcommand{\nn}{\nonumber}

\theoremstyle{definition}

\newcommand{\ds}[1]{\displaystyle #1}

\newcommand{\C}{{\mathbb C}}
\newcommand{\Z}{{\mathbb Z}}
\newcommand{\Q}{{\mathbb Q}}

\newcommand{\bs}{\boldsymbol}

\newcommand{\gl}{\mathfrak{gl}}
\newcommand{\osp}{\mathfrak{osp}}

\newcommand{\sln}{\mathfrak{sl}}

\newcommand{\la}{\lambda}

\newcommand{\sss}{\textsf{s}}

\newcommand{\ssx}{\textsf{x}}

\newcommand{\on}{\operatorname}
\newcommand{\mc}{\mathcal}

\newcommand{\cont}[2]{\contraction[1ex]{}{#1}{}{#2} #1 #2}

\begin{document}

\begin{title}[Extensions of deformed $W$-algebras]
{Extensions of deformed $W$-algebras via $qq$-characters}
\end{title}
\author{B. Feigin, M. Jimbo, and E. Mukhin}

\address{BF: National Research University Higher School of Economics,  101000, Myasnitskaya ul. 20, Moscow,  Russia, and Landau Institute for Theoretical Physics, 142432, pr. Akademika Semenova 1a, Chernogolovka, Russia
}
\email{bfeigin@gmail.com}
\address{MJ: Department of Mathematics,
Rikkyo University, Toshima-ku, Tokyo 171-8501, Japan}
\email{jimbomm@rikkyo.ac.jp}
\address{EM: Department of Mathematics,
Indiana University Purdue University Indianapolis,
402 N. Blackford St., LD 270, 
Indianapolis, IN 46202, USA}
\email{emukhin@iupui.edu}


\begin{abstract} 
We use combinatorics of $qq$-characters to study extensions of deformed $W$-algebras. We describe additional currents and part of the relations in the cases of $\gl(n|m)$ and $\osp(2|2n)$.
\end{abstract}


\maketitle

\section{Introduction}
The principal $W$-algebras form an important class of examples of Vertex Operator Algebras parameterized by root systems which enjoyed a lot of attention. It is well known that the $W$-algebras have interesting deformations, \cite{FR, KP1, KP2, FJMV}, the most famous example being the $\sln_2$ example of 
the deformed Virasoro algebra. The deformed $W$-algebras are described in a bosonic form by currents given by explicit sums of vertex operators which commute with a system of screening operators. 

The screening operators are integrals of screening currents.  The screening currents are also given by vertex operators. The contractions of screening currents are described by a deformed Cartan matrix. It turns out that the deformed $W$-algebra contains, among others, currents organized according to the representations of the corresponding quantum affine algebra. The structure of such currents is combinatorially described by $qq$-characters, which effectively encode the contractions of each term with all screening currents, \cite{N, KP1, FJM}. The $qq$-characters greatly simplify handling sums of vertex operators which commute with screening operators. In this paper we illustrate the advantages of using the $qq$-characters by studying the extensions of the deformed $W$-algebras. 

The extensions of $W$-algebras of type $\sln(n|1)$ have been studied in \cite{FS}. In that work, the $W$-algebras of type $\sln(n|1)$ were complemented by two more currents, $E(z)$ and $F(z)$, which also commute with the screening operators. In contrast to the $W$-algebra currents, $E(z)$ and $F(z)$ have nontrivial momenta. Still, the commutator $[E(z),F(w)]$ is in the $W$-algebra. For $n=2$ the resulting extension is just $\hat\sln_2$ and for $n=3$ it is the Bershadsky-Polyakov algebra, \cite{B, P}.

On the deformed side, the origin of such extension is transparent from the combinatorial point of view. Consider the example of $\gl(2|1)$. 
Let the first root be fermionic and the second one bosonic. We have a family of $4$-dimensional Kac modules generated from the trivial representation of the even part. This family depends on a parameter $\alpha$ which is the highest weight component corresponding to the first root. These modules are lifted to the quantum affine algebra as evaluation modules. The corresponding $qq$-character which  describe the corresponding currents in the deformed $W$-algebra reads:
$$
Y_1^{-1}((qs_1)^{2\alpha} z)\left(Y_1(z)+Y_1(q^2s_1^2z)Y_2(qs_1z)+
Y_1(q^2z)Y_2^{-1}(q^3s_1)+Y_1(q^4s_1^2)\right).
$$ 
This $qq$-character explicitly factorizes. This occurs because the structure of the module does not depend on $\alpha$. Then each factor produces a current which commutes with the screening operators. Using the factors we get the $E(z)$ and $F(z)$ currents:
$$
\left(Y_1(z)+Y_1(q^2s_1^2z)Y_2(qs_1z)+
Y_1(q^2z)Y_2^{-1}(q^3s_1)+Y_1(q^4s_1^2)\right)\mapsto E(z), \qquad  Y_1^{-1}(z) \mapsto  F(z).
$$
Thus $E(z)$ is a sum of 4 terms and $F(z)$ is a single term.

If we chose another Dynkin diagram  of $\gl(2|1)$ type such that both roots are fermionic, then the corresponding $qq$-character would factorize in the form:
$$
\left(Y_1(z)Y_2^{-1}(s_1z)+Y_1(q^2s_1^2z)Y_2^{-1}(q^2s_1z)\right) \left(Y_2(cz)Y_1^{-1}(cs_1z)+Y_2(cq^2s_1^2z)Y_1^{-1}(cq^2s_1z)\right).
$$
where $c=(s_1q_1)^{-2\alpha}$.
Thus, in this realization, the currents $E(z)$ and $F(z)$ both have 2 terms.

In both cases, the algebra generated by the currents 
$E(z)$ and $F(z)$ together with an extra boson
coincides with quantum affine $U_{s_3}(\widehat{\sln}_2)$
(in Wakimoto realization in the second case). 
Thus the extended deformed $W$-algebra of $\gl(2|1)$ is  the
quantum affine $\sln_2$.

Such a phenomenon happens for $\gl(n|m)$ with $m>0$ and $n>0$, and for $\osp(2|2n)$. Thus we obtain a family of new algebras generated by two currents 
$E(z)$, $F(z)$, and an extra boson, 
which contain the $W$-algebras of the corresponding kind.

We conjecture that the extended deformed $W$-algebra does not depend on the choice of a Dynkin diagram.
In the case of $\gl(n|m)$, we expect that at integer level $k$, the extended deformed $W$-algebra has an "integrable" quotient which is a deformation of the image of the $W$-algebra of $\sln(k)$ type acting in the 
$(n+k,m+k)$ minimal model extended by the primary fields 
$\phi_{(n+k-1)\omega_1,0},\phi_{(n+k-1)\omega_{k-1},0}$ 
multiplied by an appropriate Heisenberg algebra. Here 
$\omega_1,\omega_{k-1}$ 
are the first and the last fundamental $\sln_k$ weights.

\medskip

The relations between the currents in deformed $W$-algebras are elliptic, \cite{K1, K2}. However, using additional bosons, one can "undress" the currents and make the relations rational. In this paper we treat only the cases of $\gl(2n|1)$ and $\osp(2|2n)$ in the symmetric choice of the Dynkin diagrams, though the same method can be applied in other cases. The procedure of undressing is not canonical, we make  a choice and then we compute the quadratic relations of types $EE$, $FF$ which are the same as in quantum affine $\sln_2$, see Theorems \ref{prop:EE}, \ref{prop:EE osp}.  We give commutators of $E$ and $F$ with the simplest current of the deformed $W$-algebra, see Theorems \ref{lem:TETF}, \ref{lem:TETF osp}.  For $\gl(2n|1)$, we also compute the commutator $[E,F]$, and find the deformed $W$-currents corresponding to the fundamental representations in the residues, see Theorem \ref{[EF] lem}. These computations also are significantly simplified with the use of $qq$-characters. 

\medskip

We would like to refer the reader to work \cite{H} where the extended deformed $W$-algebra of type $\gl(n|1)$ has been defined and studied. We do our computations in the symmetric Dynkin diagram, while \cite{H} uses a different choice. Also, some relations we give are not written in \cite{H} and vice versa. However, 
we see 
our main contribution in the systematic use of the $qq$-characters which clarify many formulas and constructions and make it easier to understand and generalize. As a result we discover similar extensions in the cases of $\gl(n|m)$ for all $m,n$, $mn\neq 0$, and $\osp(2,2n)$. $n>1$.

\medskip

There are many questions and open problems around extended deformed $W$-algebras. The complete set of relations is not computed. The coset construction similar to \cite{FS} for extended $W$-algebras is not worked out. The representation theory is completely unknown. The conformal limit is not understood. 

\medskip

The paper is organized as follows. First, we discuss the $qq$-characters in Section \ref{qq sec}. Then we give the bosonizations for the case of $\gl(2n|1)$ in Sections \ref{subsec:boson}, \ref{boson gl(2n|1) sec}. We study the relations between various $\gl(2n|1)$ currents in  Section \ref{subsec:rational}. Section \ref{boson osp sec} contains our results in the case of $\osp(2|2n)$.
In Section \ref{boson glnm sec}, we discuss the generating current of the extended deformed $W$-algebra of type $\gl(n|m)$.

\section{The $qq$-characters}\label{qq sec}
The $qq$-characters is a combinatorial tool which can be used to construct sums of vertex operators which commute with a system of screening operators, see \cite{N, KP1, KP2, FJMV, FJM}. 
Implicitly $qq$-characters appeared already in \cite{FR}, \cite{BP}.
In this section we discuss the definition and the examples of the $qq$-characters.

\subsection{The generalities of the $qq$-characters}\label{gen qq sec}
Let $R=\Z[s_1^{\pm 1},\dots,s_t^{\pm1}]$ be a ring of Laurent polynomials in variables $s_1,\dots,s_t$. A monomial $\sigma\in R$ is a product of the form $\prod_{i=1}^ts_i^{a_i}$, where $a_i\in\Z$. 
The monomials in $R$ form a group. Note that by the definition, an integer multiple of a monomial is not a monomial.

We start with a deformed Cartan matrix. 

We call an $l\times l$ matrix $C=(c_{ij})$ a deformed Cartan matrix if it has the following form.
\begin{itemize}
    \item Each entry is a finite alternating sum of monomials: $c_{ij}=\sum_a \sigma_{ij,a}-\sum_b \sigma_{ij,b}\in R$ where, $\sigma_{ij,a},\sigma_{ij,b}$ are distinct monomials.
    \item For each $i\in\{1,\dots,l\}$ we have either $c_{ii}=\sigma_i-\sigma_i^{-1}$ (fermionic root) or $c_{ii}=\sigma_i+\sigma_i^{-1}$ (bosonic root), where $\sigma_i\in R$ are monomials.
\item There exist $d_i\in R$, $i=1,\dots,l$, such that the matrix $B=(d_ic_{ij})$ is symmetric. Moreover, for bosonic roots, $d_i$ have the form   $d_i=-(\sigma_i'-(\sigma_i')^{-1})(\sigma_i''-( \sigma_i'')^{-1})$ where $\sigma_i,\sigma_i''$ are monomials such that $\sigma_i\sigma_i'\sigma_i''=1$, and for fermionic roots we have $d_i=c_{ii}=\sigma_i-\sigma_i^{-1}$.
\item $\det C\neq 0$. 
\end{itemize}

Note, that by definition,  $\sigma_i$  (and $\sigma_i'$, 
$\sigma_i''$) which determine the diagonal entry $c_{ii}$ and the symmetrizing factor $d_i$, are fixed for each $i$.

We often call elements of the set $\{1,\dots, l\}$ labeling the rows and columns of the deformed Cartan matrix "colors". For each color we will have a root.

\medskip 

Some examples of deformed Cartan matrices with only bosonic roots are given in \cite{FR}. The Cartan matrices with only fermionic roots and such that all diagonal entries are the same were studied in \cite{FJM}. A number of examples of deformed Cartan matrices is given in Appendix A of \cite{FJMV}, see also \eqref{gl(2n|1) cartan}, \eqref{osp cartan} below.

Note that there are important examples of deformed Cartan matrices which do not fit the definition above. That includes, in classification of \cite{FJMV}, type $A$, $(1,2,3)$ where $\det C=0$ and type $B$ $(1,2,2)$
where there is a diagonal entry of the form $q-1+q^{-1}$.

\medskip

Given a deformed Cartan matrix $C$, we define the roots.

First, we prepare some notation and terminology.
Let $\mc Y_l=\Z[Y^{\pm1}_{i,\sigma}]$ be the ring of Laurent polynomials in variables $Y_{i,\sigma}$ where $i=1,\dots, l$, and $\sigma\in R$ runs over all monomials. A monomial $m \in\mc Y_l$ is a finite product of the generators $\prod_{j=1}^sY_{i_j,\mu_j}^{a_j}$, where $a_j\in\Z$. 
The monomials in $\mc Y_l$ form a group. 

A monomial $m\in\mc Y_l$ is called generic if it is a finite product of distinct generators, i.e. if all non-trivial powers are one or minus one,  $a_j=\pm1$.

Two monomials $m,n\in \mc Y_l$ are called mutually generic if generators $Y_{i,\sigma}$ present in $m$ are not present in $n$, i.e. if monomials $mn$ and $m/n$ are generic.

A Laurent polynomial $\chi\in\mc Y_l$ is a finite sum of monomials with integer coefficients, $\chi=\sum_m a_mm$. We say $m\in\chi$ if and only if $a_m\neq 0$. We call Laurent polynomials $\chi_1,\ \chi_2\in \mc Y_l$ mutually generic if for all $m\in\chi_1$, $n\in\chi_2$, the monomials $m,n$ are mutually generic.

For a monomial $\mu \in R$, let $\tau_\mu: \mc Y_l \to \mc Y_l$ be the shift automorphism sending $Y_{i,\sigma}\mapsto Y_{i,\mu \sigma}$.

For $i\in\{1,\dots, l\}$, let $\rho_i: \mc Y_l \to \mc Y_1$ be the restriction homomorphism of rings sending $Y_{i,\sigma}\mapsto Y_{1,\sigma}$ and   $Y_{j,\sigma}\mapsto 1$, for $j\neq i$.

\medskip

Let $C=(c_{ij})$ where $c_{ij}=\sum_a \sigma_{ij,a}-\sum_b \sigma_{ij,b}\in R$ be a deformed Cartan matrix.
For $i=1,\dots,l$, and a monomial $\mu\in R$, define the affine root $A_{i,\mu}\in\mc Y$ by  
\begin{align*}
A_{i,1}=\prod_{j=1}^l\prod_a Y_{j,\sigma_{ij,a}}\prod_b Y_{j,\sigma_{ij,b}}^{-1}, \quad  {\rm and} \quad A_{i,\mu}=\tau_\mu (A_{i,1}).
\end{align*}
Note that since $\det C\neq 0$, the affine roots $A_{i,\mu}$ are all algebraically independent.

We often denote $Y_{i,\sigma}$ by $\bs i_\sigma$, $Y_{i,\sigma}^{-1}$ by $\bs i^\sigma$, $Y_{i,\sigma}Y_{i,\mu}$ by  $\bs i_{\sigma,\mu}$, etc.

\medskip

Next, we define basic $qq$-characters in the case $l=1$.

We start with the definition of elementary blocks in the fermionic case. We have $C=(q-q^{-1})$ for some monomial $q\in R$. Recall that we denote $Y_{1,\sigma}$ by $\bs 1_\sigma$, $Y_{1,\sigma}^{-1}$ by $\bs 1^\sigma$, $Y_{1,\sigma}Y_{1,\mu}$ by  $\bs 1_{\sigma,\mu}$, etc. In particular, we have $A_{1,1}=\bs 1_q^{q^{-1}}$.

An elementary block $B^{(k)}\in\mc Y_1$ of length $k+1$ is the sum of $k+1$ monomials of the form
\begin{align*}
B^k_\mu&=m\tau_\mu(\bs 1_{q^{2k-2}\hspace{-3pt},\dots,q^2,1}+\bs 1_{q^{2k-2}\hspace{-3pt},\dots,q^4,q^2,q^{-2}}+ \bs  1_{q^{2k-2}\hspace{-3pt},\dots,q^4,1,q^{-2}}+\dots + \bs 1_{q^{2k-4}\hspace{-3pt},\dots,q^2,1,q^{-2}}) \\
&=m\tau_\mu\big(\bs 1_{q^{2k-2}\hspace{-3pt},\dots,q^2,1}(1+A^{-1}_{1,q}+ A^{-1}_{1,q} A^{-1}_{1,q^3}+\dots + A^{-1}_{1,q} A^{-1}_{1,q^3}\dots A^{-1}_{1,q^{2k-1}})\big),
\end{align*}
where $\mu\in R$ is an arbitrary monomial and $m\in\mc Y_1$ is a monomial of the form $\bs 1^{\nu_1,\dots,\nu_s}$, where $\nu_i\in R$ are monomials and $\nu_i/\mu\neq q^{-2},1,q^2,\dots, q^{2k-2}$ for all $i=1,\dots,s$.

We continue with the definition of elementary blocks in the bosonic case. We have $C=(q+q^{-1})$, $d_1=-(\sigma_1-\sigma_1^{-1})(\sigma_2-\sigma_2^{-1}) $ for some monomials $q,\sigma_1,\sigma_2\in R$ with $q\sigma_1\sigma_2=1$. We have $A_{1,1}=\bs 1_{q,q^{-1}}$.
In the bosonic case we define two types of elementary blocks $B^{(k)}\in\mc Y_1$ of length $k+1$:
\begin{align*}
B^{i,k}_\mu&=\tau_\mu(\bs 1_{\sigma_j^{2k-2}\hspace{-3pt},\dots,\sigma_j^{2},1}+\bs 1_{\sigma_j^{2k-2}\hspace{-3pt},\dots,\sigma_j^4,\sigma_j^2}^{q^2}+ \bs  1_{\sigma_j^{2k-2}\hspace{-3pt},\dots,\sigma_j^4}^{\sigma_j^2q^2,q^2}+\dots + \bs 1^{\sigma_j^{2k-2}q^2\hspace{-3pt},\dots,\sigma_j^2q^2,q^2}) \\
&=\tau_\mu\big(\bs 1_{\sigma_j^{2k-2}\hspace{-3pt},\dots,\sigma_j^2,1}(1+A^{-1}_{1,q}+ A^{-1}_{1,q} A^{-1}_{1,q\sigma_j^2}+\dots + A^{-1}_{1,q} A^{-1}_{1,q\sigma_j^2}\dots A^{-1}_{1,q\sigma_j^{2k-2}})\big),
\end{align*}
where $j=1,2$.

Now we are ready to define basic $qq$-characters.
We say that $\chi\in\mc Y_1$ is a basic $qq$-character if $\chi$
is a sum of products of mutually generic elementary blocks. That is
$\chi=\sum_{j=1}^a \prod_{s=1}^{b_j} B_{sj}$, where all $B_{sj}$ are elementary blocks and $B_{sj}$, $B_{s'j}$ are mutually generic for all $s,s',j$. 

A basic $qq$-character $\chi$  is tame in terminology of \cite{FJM}, meaning that all monomials $m\in\chi$ are generic.
In the bosonic case, there are other tame $qq$-characters which are not basic. For example, there exists a $5$ terms tame $qq$-character:
$$
\bs 1_{1,\sigma_1^{-2},\sigma_2^{-2}}+\bs 1_{1,\sigma_2^{-2}}^{q^2\sigma_1^{-2}}+\bs 1_{1,\sigma_1^{-2}}^{q^2\sigma_2^{-2}}+\bs 1_{1}^{q^2\sigma_1^{-2},q^2\sigma_2^{-2}}+\bs 1^{q^2,q^2\sigma_1^{-2},q^2\sigma_2^{-2}}.
$$
We do not consider non-basic $qq$-characters in this paper and we hope to return to their study in  future publications.

\medskip

Finally, we define the basic $qq$-characters in the general case.
Recall the affine roots $A_{i,m}$.  We say that $\chi$ is an  elementary block of color $i$ and length $k+1$ 
if the restriction  $\rho_i(\chi)$  of $\chi$ to color $i$ is an elementary block of length $k+1$ and 
if for some monomial $\mu\in R$ and some monomial  $m\in\mc Y_l$
$$
\chi=\tau_\mu\big(m(1+A^{-1}_{i,q}+ A^{-1}_{i,q} A^{-1}_{i,q^3}+\dots + A^{-1}_{i,q} A^{-1}_{i,q^3}\dots A^{-1}_{i,q^{2k-1}})
\big)
$$
in the fermionic case $c_{ii}=q-q^{-1}$, or if 
$$
\chi=\tau_\mu\big(m(1+A^{-1}_{i,q}+ A^{-1}_{i,q} A^{-1}_{i,q\sigma_j^2}+\dots + A^{-1}_{i,q} A^{-1}_{i,q\sigma_j^2}\dots A^{-1}_{i,q\sigma_i^{2k-2}})
\big)
$$
in the bosonic case $c_{ii}=(q+q^{-1})$, $d_i=-(\sigma_1-\sigma_1^{-1})(\sigma_2-\sigma_2^{-1})$, and $j=1,2$.

We say that $\chi\in\mc Y_l$ is a basic $qq$-character if for any $i=1,\dots, l$,  $\chi$
can be written as a sum of products of mutually generic elementary blocks of color $i$. That is for each $i$ we can write
$\chi=\sum_{j=1}^{a^{(i)}} \prod_{s=1}^{b_j^{(i)}} B_{sj}^{(i)}$, where all $B_{sj}^{(i)}$ are elementary blocks of color $i$ and $B_{sj}^{(i)}$, $B_{s'j}^{(i)}$ are mutually generic for all $s,s',j$.

For the most part, the $qq$-characters will be sums of elementary blocks of length $1$ or $2$, and sometimes $3$. Moreover, most commonly we will have $b_j^{(i)}=1$.

All $qq$-characters in this paper are basic, so we simply call them the $qq$-characters.

 The first monomial in an elementary block we call the top monomial. This is the unique monomial such that all other monomials in the elementary block are obtained from it by multiplication by inverse roots $A_{i,\sigma}^{-1}$. We call a monomial $m$ in a $qq$-character $\chi$ a top monomial, if for any color $i$, the monomial $m$ is the product of top monomials in elementary blocks of color $i$,  $B_{sj}^{(i)}$ for some $j$ and all $s=1,\dots,b_j$.

Here we consider only $qq$-characters which are Laurent polynomials (with finitely many monomials). In this case, it is easy to see that any $qq$-character has a top monomial. 

Conversely, given a top monomial, one often can 
reconstruct the $qq$-character, recursively by adding the other monomials in elementary blocks.
In fact, we often use this method to obtain the $qq$-characters but in the end we simply state the results and show they are correct. We do not discuss the details, the idea of such procedure is well known, see \cite{FM, FJM}.

\medskip

For each $i=1,\dots,l$, define a $\Z$-grading $\deg_i$ of $\mathcal Y_l$ by setting $\deg_i Y_{j,\sigma}^{\pm1}=0$ if $i$ is bosonic and
$\deg_i Y_{j,\sigma}^{\pm1} =\pm\delta_{ij}$ if $i$ is fermionic. We write $\deg m=(\deg_i(m))_{i=1,\dots,l}$ and call it the degree of $m\in\mathcal Y_l$. 

The currents in the deformed $W$-algebras correspond to $qq$-characters of degree zero. The main idea of this paper to find $qq$-characters corresponding to the $qq$-characters of non-zero degree and add the corresponding currents to the deformed $W$-algebras.

\medskip

The deformed Cartan matrices in this text will depend only on two parameters. We have $R=\Z[s_1^{\pm1},s_2^{\pm1}]$ and we set
$$
s_3=(s_1s_2)^{-1}, \qquad q=s_2, \qquad t_i=s_i-s_i^{-1} \qquad (i=1,2,3).
$$

\subsection{The case of $\gl(2n|1)$}\label{gl(2n|1) sec}
In this section we describe some $qq$-characters related to the deformed $W$-algebra of type $\gl(2n|1)$ in the symmetric parity.

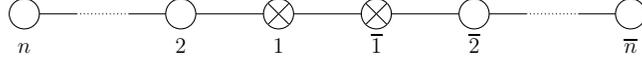
\begin{figure}
\begin{align*}
\begin{tikzpicture}
\dynkin[root radius=.2cm, edge length=1.3cm, 
labels={n,2,1,\overline{1},\overline{2},\overline{n}}]{A}{o.otto.o}
\end{tikzpicture}
\end{align*}
\caption{The $\gl(2n|1)$ Dynkin diagram and labeling.}\label{gl fig}
\end{figure}

We consider the Dynkin diagrams for the Lie superalgebra $\gl(2n|1)$ with the symmetric location of the fermionic roots and label roots  from $n$ to $1$ and then from $\bar 1$ to $\bar n$, such that the fermionic roots are $1,\bar 1$ and the bijection $i  \leftrightarrow \bar i$ is an involution of the Dynkin diagram, see Figure \ref{gl fig}.  

Let $s_1,s_2=q,s_3$ satisfy $s_1s_2s_3=1$.
We study the deformed $2n\times 2n$ Cartan matrix whose non-trivial entries are given by:
\begin{itemize}
\item
$c_{ii}=q+q^{-1}$ $(i\neq 1,\bar 1)$, $c_{11}=c_{\bar1\bar1}=t_3$,

\item
$c_{1\bar1}=c_{\bar1 1}=t_2$, $c_{12}=c_{\bar 1\bar 2}=t_1$
\item
$c_{i,i+1}=c_{\bar{i},\overline {i+1} }=-1$ $(i=2,\dots,n-1)$.
\item
$c_{i+1,i}=c_{ \overline{i+1},\bar{i}}= -1$ $(i=1,\dots,n-1)$.
\end{itemize}
We have $d_i=-t_1t_3$  $(i\neq 1,\bar 1)$, and $d_1=d_{\bar 1}=t_3$.

As an example we write the matrix for the case $\gl(6|1)$:
\begin{align}\label{gl(2n|1) cartan}
C_{\gl(6|1)}=\begin{pmatrix}
q+q^{-1} & -1 & 0 & 0 & 0 & 0 \\
-1 & q+q^{-1} & -1 & 0 & 0 & 0 \\
0 & s_1-s_1^{-1} & s_3-s_3^{-1} & q-q^{-1} & 0 & 0  \\
0 & 0 & q-q^{-1} & s_3-s_3^{-1} & s_1-s_1^{-1} & 0 \\
0& 0 & 0 & -1 & q+q^{-1} & -1 \\
0& 0& 0 & 0 & -1 & q+q^{-1}
\end{pmatrix}.
\end{align} 

Note the symmetry $i\leftrightarrow \bar i$. In this section, this symmetry is always preserved. Any formula has an analog where all colors are replaced  by that rule.

In the classification of \cite{FJMV}, our matrix is of type $(2,\dots,2,1,1,2,\dots,2)$A and it corresponds to the product of Fock spaces $\mc F_2^{\otimes n} \otimes \mc F_1 \otimes \mc F_2^{\otimes n}$ of the quantum toroidal $\gl_1$ algebra.

We recall the notation $Y_{i,\mu}=\bs i_\mu$, $Y_{i,\mu}^{-1}=\bs i^\mu$,  
$Y_{i,\mu}Y_{i,\nu}Y^{-1}_{i,\kappa}=\bs i_{\mu,\nu}^\kappa$, etc.

Then the roots have the form
\begin{align*}
A_{n,1}&=\bs n_{q,q^{-1}} (\bs{n-1})^1, \qquad A_{i,1}=\bs i_{q,q^{-1}} (\bs{i-1})^1(\bs{i+1})^1 \qquad (i=3,\dots, n-1), \\ A_{2,1}&=\bs 2_{q,q^{-1}}\bs 3^1\bs 1_{s_1}^{s_1^{-1}},\qquad\quad \   A_{1,1}=\bs 1_{s_3}^{s_3^{-1}} \bs {\bar 1}_q^{q^{-1}}\bs 2^1.
\end{align*}
The roots $A_{\bar i,1}$ are given by the symmetry $i\leftrightarrow \bar i$.

\medskip

In the case we consider, one expects to find $qq$-characters of degree zero which correspond to finite-dimensional modules of quantum affine algebra $U_q\gl(2n|1)$. We describe some of them now.

We have the $qq$-character $\chi_{1,1}$ corresponding to the vector $(2n+1)$-dimensional representation of $\gl(2n|1)$. It is described as follows.
\begin{align}\label{one box form}
\chi_{1,1}=\bs n_1+\bs n^{q^2}(\bs{n-1})_q+(\bs{n-1})^{q^3}(\bs{n-2})_{q^2}+\dots +
\bs 3^{q^{n-1}}\bs 2_{q^{n-2}}+\bs 2^{q^{n}}\bs 1^{q^{n-1}s_1}_{q^{n-1}s_1^{-1}}+ \\
+\bs 1^{q^{n-1}s_1}_{q^{n+1}s_1}\bs{\bar 1}_{q^{n-1}}^{q^{n+1}}+ 
\bs{\bar 1}^{q^{n+1}}_{q^{n+1}s_1^{2}}\bs{\bar 2}_{q^{n}s_1}
+\bs{\bar 2}^{q^{n+2}s_1}\bs{\bar 3}_{q^{n+1}s_1}+\dots + \bs{\bar n}^{q^{2n}s_1}.\notag
\end{align}
We also set $\chi_{1,\mu}=\tau_\mu(\chi_{1,1})$.

Note that $q^{2n} s_1$ is the central charge of the product $\mc F_2^{\otimes n} \otimes \mc F_1 \otimes \mc F_2^{\otimes n}$.

The $qq$-character $\chi_{1,1}$ corresponds to the representation of $\gl(2n|1)$ which is in the standard way denoted by a Young diagram with a single box. 
Then the $2n+1$ terms of the $qq$-character correspond to a basis of this representation labeled by the semi-standard Young tableaux with the alphabet $\{n,n-1,\dots, 1, 0, \bar 1, \dots,\bar n\}$. 
One can view this representation as described on Figure \ref{one box pic}. 

We chose the an order of the alphabet $\{n,n-1,\dots,1,0,\bar 1,\dots,\bar n\}$ along the arrows on Figure \eqref{one box pic}: 
$n\prec n-1\prec n-2\prec \dots\prec 1\prec 0\prec \bar 1\prec \dots \prec\bar n$.

\begin{figure}
\begin{align*}
\begin{tikzpicture}
\node[rectangle,draw,minimum width = 0.75cm, 
    minimum height = 0.75cm] (r) at (-9,0) {};
 \node[rectangle,draw,minimum width = 0.75cm, 
    minimum height = 0.75cm] (r) at (-6.75,0) {};
    \node[rectangle,draw,minimum width = 0.75cm, 
    minimum height = 0.75cm] (r) at (-2.25,0) {};
    \node[rectangle,draw,minimum width = 0.75cm, 
    minimum height = 0.75cm] (r) at (0,0) {};
    \node[rectangle,draw,minimum width = 0.75cm, 
    minimum height = 0.75cm] (r) at (2.25,0) {};
    \node[rectangle,draw,minimum width = 0.75cm, 
    minimum height = 0.75cm] (r) at (6.75,0) {};
    \node at (-9,0) {\footnotesize$n$};
    \node at (-6.75,0){\footnotesize$n\hspace{-3pt}-\hspace{-3pt}1$};
    \node at (-4.5,0) {$\dots$};
    \node at (4.5,0) {$\dots$};
    \node at (-2.25,0) {\footnotesize$1$};
        \node at (0,0) {\footnotesize$0$};
        \node at (2.25,0) {\footnotesize$\overline{1}$};     
        \node at (6.75,0) {\footnotesize$\overline{n}$};  
\draw[ ->] (-8.5,0)-- node[below]{{\small $A_{n,q}^{-1}$}}(-7.25,0);
\draw[ ->] (-6.25,0)-- node[below]{{\small $A_{n-1,q^2}^{-1}$}}(-5,0);
\draw[ ->] (-4,0) -- node[below]{{\small $A_{2,q^{n-1}}^{-1}$}} (-2.75,0);
\draw[ ->] (-1.75,0)-- node[below]{{\small $A_{1,q^n}^{-1}$}}(-0.5,0);
\draw[ ->] (0.5,0)-- node[below]{\small $A^{-1}_{\overline{1},q^ns_1}$}(1.75,0);
\draw[ ->] (2.75,0)-- node[below]{\hspace{5pt}\small $A^{-1}_{\overline{2},q^{n+1}s_1}$}(4,0);
\draw[ ->] (5,0)-- node[below]{\small $A^{-1}_{\overline{n},q^{2n-1}s_1}\hspace{2pt}$}(6.25,0);
\end{tikzpicture}
\end{align*}
\caption{The $qq$-character corresponding to the $\gl(2n|1)$ vector representation.}\label{one box pic}
\end{figure}

Next we describe the $qq$-character $\chi_{k,1}$ corresponding to the $k$-th skew-symmetric power of the vector representation of $\gl(2n|1)$ $(k=1,\dots,n)$. 
This representation corresponds to one column Young diagram with $k$ boxes.

A 
semi-standard Young tableau is a 
filling of the $k$ boxes in the column with elements of the alphabet $\{n,n-1,\dots,1,0,\bar 1,\dots,\bar n\}$ 
in non-decreasing order from up down. The fermionic filling $0$ is allowed to repeat. The other fillings are bosonic and cannot repeat. Now we describe the monomial corresponding to each filling.

The top term of  $\chi_{k,1}$ corresponding to the minimal filling $n,n-1,\dots,n+1-k$ is $(\bs{n+1-k})_1$. 

In general all monomials $\chi_{k,1}$ are inside of the product $\chi_{1,q^{k-1}}\chi_{1,q^{k-3}}\dots \chi_{1,q^{-k+1}}$. 
Let $M_{i,1}$ be the monomial of $\chi_{1,1}$ corresponding to the filling of the box with 
$i\in\{n,n-1,\dots,1,0,\bar 1,\dots, \bar n\}$ 
and $M_{i,\mu}=\tau_\mu(M_{i,1})$.  Then the monomial in $\chi_{k,1}$ corresponding to the filling of the column with 
$k$-boxes $i_1\preceq i_2\preceq \dots\preceq i_k$ 
is
$$
M_{i_1,\dots,i_k}=M_{i_1,q^{k-1}}M_{i_2,q^{k-3}} \dots M_{i_k,q^{-k+1}}.
$$

We have the following lemma. 
\begin{lem}\label{gl 2n one column lem}
The sum
$$
\chi_{k,1}=\sum_{\substack{i_1,\dots,i_k\in\{n,\dots,1,0,\bar 1,\dots,\bar n\} \\ 
i_1\preceq i_2\preceq \dots\preceq i_k}}M_{i_1,\dots,i_k},
$$
where the equality $i_s=i_{s+1}$ is allowed only if $i_s=i_{s+1}=0$, is a degree zero basic $qq$-character corresponding to the deformed Cartan matrix of $\gl(2n|1)$ type, cf. \eqref{gl(2n|1) cartan}.
\end{lem}
\begin{proof}
For $i\in \{n,\dots,1,0,\bar 1,\dots,\bar n\}$, let 
$i^+$  be the next element. 
We need to check the decomposition of $\chi_{k,1}$ into elementary blocks of color $i$.

If $i_s^+\neq i_{s+1}$ 
or if $s=k$ and $i_s\neq \bar n$ or if $i_s=1$, then 
$M_{i_1,\dots,i_s,\dots,i_k}+ M_{i_1,\dots,i_s^+,\dots,i_k}$ 
is an elementary block of color $i_s$ if $i_s\prec 0$,
or $i_s^+$ if $i_s\succeq 0$.
The other cases correspond to elementary blocks of length $1$.
\end{proof}

The number of terms in $\chi_{k,1}$
is clearly equal to the dimension of the $k$-th skew symmetric power of the space of signature $(2n,1)$:
$$
\dim \bigwedge^k
(\C^{2n|1})=\sum_{i=0}^k {2n \choose i}.
$$

We also have the corresponding degree zero basic $qq$-characters $\chi_{\bar k,1}$ obtained from $\chi_{k,1}$ by changing colors $i\leftrightarrow \bar i$.

\medskip

Now we exhibit a $qq$-character $\xi_\mu$ of degree $(0,\dots,0,1,-1,0,\dots,0)$ which has $2^n$ terms.
We have $\xi_\mu=\tau_\mu(\xi_1)$ and 
$\xi_1$ is defined recursively as follows. 

To describe the recursion, we explicitly show the dependence on $n$ and write $\xi_\mu=\xi_\mu^{(n)}$.
We have  
$$
\xi_1^{(1)}=\bs 1_1 \bs{\bar 1}^{s_1}+\bs 1_{q^2s_1^2}\bs{\bar 1}^{q^2s_1},
$$
$$
\xi_1^{(2)}=\bs 1_1 \bs{\bar 1}^{s_1}+\bs 1_{q^2s_1^2}\bs 2_{q s_1}\bs{\bar 1}^{q^2s_1}+ 
\bs 1_{q^2} \bs 2^{q^3s_1}\bs{\bar 1}^{q^2s_1}+
\bs 1_{q^4s_1^2}\bs{\bar 1}^{q^4s_1},
$$
\begin{align*}
\xi_1^{(3)}=\bs 1_1 \bs{\bar 1}^{s_1}+\bs 1_{q^2s_1^2}\bs 2_{qs_1}\bs{\bar 1}^{q^2s_1}+ 
\bs 1_{q^2} \bs 2^{q^3s_1}\bs 3_{q^2s_1} \bs{\bar 1}^{q^2s_1}+
\bs 1_{q^4s_1^2}\bs 3_{q^2s_1}\bs{\bar 1}^{q^4s_1}+\\
+\bs 1_{q^2}\bs 3^{q^4s_1} \bs{\bar 1}^{q^2s_1}+\bs 1_{q^4s_1^2}\bs 2_{q^3s_1}\bs 3^{q^4s_1}\bs{\bar 1}^{q^4s_1}+ 
\bs 1_{q^4} \bs 2^{q^5s_1} \bs{\bar 1}^{q^4s_1}+
\bs 1_{q^6s_1^2}\bs{\bar 1}^{q^6s_1}.
\end{align*}
We also picture the case $n=3$ in Figure \ref{char gl(2n|1)}.

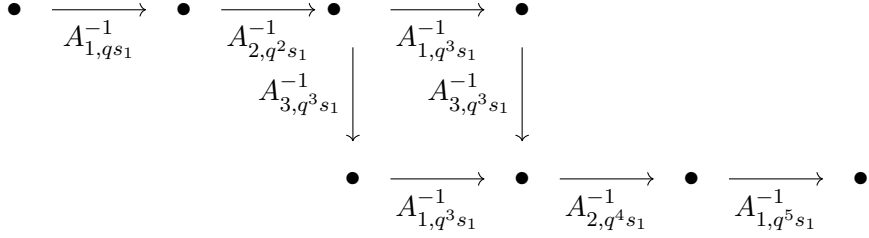
\begin{figure}
\begin{align*}
\begin{tikzpicture}
\node at (-6.75,0) {$\bullet$};
\node at (-4.5,0) {$\bullet$};
\node at (-2.5,0) {$\bullet$};
\node at (0,0) {$\bullet$};
\node at (-2.25,-2.25) {$\bullet$};
\node at (0,-2.25) {$\bullet$};
\node at (2.25,-2.25) {$\bullet$};
\node at (4.5,-2.25) {$\bullet$};
\draw[ ->] (-6.25,0)-- node[below]{{\small $A_{1,qs_1}^{-1}$}}(-5,0);
\draw[ ->] (-4,0)-- node[below]{{\small $A_{2,q^2s_1}^{-1}$}}(-2.75,0);
\draw[ ->] (-1.75,0)-- node[below]{{\small $A_{1,q^3s_1}^{-1}$}}(-0.5,0);
\draw[ ->] (-1.75,-2.25)-- node[below]{{\small $A_{1,q^3s_1}^{-1}$}}(-0.5,-2.25);
\draw[ ->] (0.5,-2.25)-- node[below]{{\small $A_{2,q^4s_1}^{-1}$}}(1.75,-2.25);
\draw[ ->] (2.75,-2.25)-- node[below]{{\small $A_{1,q^5s_1}^{-1}$}}(4,-2.25);
\draw[ ->] (-2.25,-0.5)-- node[left]{{\small $A_{3,q^3s_1}^{-1}$}}(-2.25,-1.75);
\draw[ ->] (0,-0.5)-- node[left]{{\small $A_{3,q^3s_1}^{-1}$}}(0,-1.75);
\end{tikzpicture}
\end{align*}
\caption{The $qq$-character $\xi_1^{(3)}$.}\label{char gl(2n|1)}
\end{figure}

In general we set 
\begin{align}\label{init gl(2n|1)}
\xi_\mu^{(n)}=\xi_\mu^{(n),1}+\xi_\mu^{(n),2}, \qquad 
\xi_\mu^{(1),1}=\bs 1_\mu\bs{\bar 1}^{\mu s_1}, \quad \xi_\mu^{(1),2}=\bs 1_{\mu q^2 s_1^2}\bs{\bar 1}^{\mu q^2s_1},
\end{align}
and
\begin{align}\label{rec}
\begin{aligned}
\xi^{(n),1}_1&=\xi^{(n-1),1}_1+\bs n_{q^{n-1}s_1} \xi^{(n-1),2}_1, \\
\xi^{(n),2}_1&=\bs n^{q^{n+1}s_1} \xi^{(n-1),1}_{q^2}+\xi^{(n-1),2}_{q^2}.
\end{aligned}
\end{align}

\begin{lem}\label{eta lem}
The recursion  \eqref{init gl(2n|1)}, \eqref{rec} defines a basic $qq$-character $\xi_1^{(n)}$ with $2^n$ terms corresponding to the deformed Cartan matrix of $\gl(2n|1)$ type, cf. \eqref{gl(2n|1) cartan}, of degree $(0,\dots,0,1,-1,0,\dots,0)$.
\end{lem}
\begin{proof} We proceed by induction on $n$.
Assume the statement is true for $(n-1)$. To prove it for the next case, we have to check the decomposition of $\xi_1^{(n)}$ into elementary blocks for all $i\in\{n,\dots,1,\bar 1,\dots,\bar n\}$. 

For $i=\bar 1, \dots,\bar n$, all blocks are clearly of length one. For $i=1,\dots,n-2$, the statement follows immediately from the induction hypothesis. 

Note that $\xi^{(n-1),2}=\tau_{q^2}(\bs (\bs{n-1})^{q^{n-2}s_1}\xi^{(n-1),1})$. Therefore 
$$
A^{-1}_{n,q^ns_1}\bs n_{q^{n-1}s_1} \xi^{(n-1),2}_1=\bs n^{q^{n+1}s_1} \xi^{(n-1),1}_{q^2}.
$$
In addition, clearly $\xi_1^{(n-1),1}, \xi_1^{(n-1),2}$ do not contain variable $\bs n$.
That proves the statement for $i=n$.

The statement for $i=n-1$ holds for $\xi_1^{(n),1}$ and $\xi_1^{(n),2}$ separately. Indeed, for example,
\begin{align*}
\xi_1^{(n),1}=\xi_1^{(n-2),1}+(\bs{n-1})_{q^{n-2}s_1}\xi_1^{(n-2),2}+ \bs n_{q^{n-1}s_1}(\bs{n-1})^{q^{n}s_1}\xi_{q^2}^{(n-2),1}
+\bs n_{q^{n-1}s_1}\xi_{q^2}^{(n-2),2}\\
= \xi_1^{(n-2),1}+(\bs{n-1})_{q^{n-2}s_1}\xi_1^{(n-2),2}(1+A^{-1}_{n-1, q^{n-1}s_1})
+\bs n_{q^{n-1}s_1}\xi_{q^2}^{(n-2),2}.
\end{align*}

\end{proof}

The recursion in Lemma \ref{eta lem} can be solved explicitly.

\begin{lem}\label{xi explicit lem} We have
\begin{align}\label{xi-explicit}
\xi^{(n)}_1=&\sum_{\nu\in\{0,1\}^n}\tilde\xi_{\nu},\\
&\tilde\xi_{\nu_n,\ldots,\nu_1}
=\prod_{i=2}^n Y_{i,q^{2\sum_{j=i+1}^n\nu_j+\nu_i-\nu_{i-1}+i}s_1}^{-\nu_i+\nu_{i-1}}\cdot 
Y_{1,q^{2\sum_{j=1}^n\nu_j}s_1^{2\nu_1}}
Y_{\bar 1,q^{2\sum_{j=1}^n\nu_j}s_1}^{-1}\,.\notag
\end{align}
\end{lem}
\begin{proof}
Indeed,  by Lemma \ref{eta lem}, the recursion has the form
\begin{align*}
&\tilde\xi_{0,0,\nu'}=\tilde\xi_{0,\nu'}\,,\quad \tilde\xi_{0,1,\nu'}=\tilde\xi_{1,\nu'}\bs{n}_{q^{n-1}s_1}\,,\\
&\tilde\xi_{1,0,\nu'}=\tau_{q^2}\bigl(\tilde\xi_{0,\nu'}\bigr)\bs{n}^{q^{n+1}s_1}\,,
\quad \tilde\xi_{1,1,\nu'}=\tau_{q^2}\bigl(\tilde\xi_{1,\nu'}\bigr)\,,
\end{align*}
or equivalently 
\begin{align}\label{xi-rec}
\tilde\xi_{\nu_n,\nu_{n-1},\nu'}=\tau_{q^{2\nu_n}}\bigl(\tilde\xi_{\nu_{n-1},\nu'}\bigr)
Y_{n,q^{n+\nu_n-\nu_{n-1}}s_1}^{-\nu_n+\nu_{n-1}}\,
\quad\text{for $\nu_n,\nu_{n-1}\in\{0,1\}$},
\end{align}
with the initial condition
$\tilde\xi_0=\bs{1}_1\bar{\bs{1}}^{s_1}$, $\tilde\xi_1=\bs{1}_{q^2s_1^2}\bar{\bs{1}}^{q^2s_1}$.

The lemma follows.
\end{proof}

By the symmetry $i \leftrightarrow \bar i$, using  $\xi_1$ we also obtain a basic $qq$-character with $2^n$ terms of degree $(0,\dots,0,-1,1,0,\dots,0)$ which we denote $\eta_1$.

\begin{rem}
There are many more $qq$-characters with finitely many terms and various degrees in the presence of fermionic roots. For example,
we have a $qq$-character of degree $(1,0,\dots,-1,0,\dots,0)$ with $n$ terms
$$
\bs 1^{q^{n-1}s_1^{-1}}\left(\bs n_1+\bs n^{q^2}(\bs{n-1})_q+(\bs{n-1})^{q^3}(\bs{n-2})_{q^2}+\dots +
\bs 3^{q^{n-1}}\bs 2_{q^{n-2}}+\bs 2^{q^{n}}\bs 1^{q^{n-1}s_1}_{q^{n-1}s_1^{-1}}\right).
$$
While it is an interesting problem to understand the totality of such $qq$-characters, the products of such $qq$-characters never seem to have degree zero.
For that reason we do not discuss them in this text.
\end{rem}

\subsection{The case of $\osp(2|2n)$} \label{osp(2|2n) sec}

In this section we describe some $qq$-characters related to the deformed $W$-algebra of type $\osp(2|2n)$. 

\begin{figure}
\begin{align*}
\begin{tikzpicture}
\dynkin[root radius=.2cm, edge length=1.3cm, labels={n,3,2,1,\overline{1}}]{D}{o.oott}
\draw (3.98,-0.92) to (3.98,0.92);
\draw (4.08,-0.92) to (4.08,0.92);
\end{tikzpicture}
\end{align*}
\caption{The $\osp(2|2n)$ Dynkin diagram and labeling.}\label{osp fig}
\end{figure}
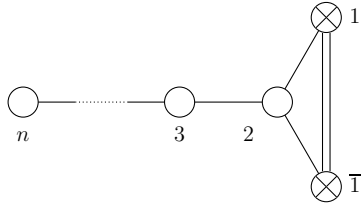

We consider the Dynkin diagrams for the Lie superalgebra $\osp(2|2n)$ with the symmetric location of the fermionic roots and label roots such that the fermionic roots are $1,\bar 1$ and the bijection $1  \leftrightarrow \bar 1$ keeping the rest of the roots is an involution of the Dynkin diagram, see Figure \ref{osp fig}. 

Let $s_1,s_2=q,s_3$ satisfy $s_1s_2s_3=1$.
We study the deformed $(n+1)\times (n+1)$ Cartan matrix whose non-trivial entries are given by:
\begin{itemize}
\item
$c_{ii}=q+q^{-1}$ $(i\neq 1,\bar 1)$, $c_{11}=c_{\bar1\bar1}=t_3$,
\item
$c_{i,i+1}=c_{i+1,i}=-1$ $(i=2,\dots,n-1)$.
\item
$c_{1\bar1}=c_{\bar1 1}=qs_1^{-1}-s_1q^{-1}$, 
\item
$c_{2,1}=c_{2,\bar 1}= -1$, $c_{12}=c_{\bar 12}=t_1$.
\end{itemize}
We have $d_i=-t_1t_3$ $(i=2,\dots,n)$, and $d_1=d_{\bar 1}=t_3$.

As an example we write the matrix for the case $\osp(2|8)$:
\begin{align}\label{osp cartan}
C_{\osp(2|8)}=\begin{pmatrix}
q+q^{-1} & -1 & 0 & 0 & 0 \\
-1 & q+q^{-1} & -1 & 0 & 0  \\
0 & -1 & q+q^{-1} & -1 & -1  \\
0 & 0 & s_1-s_1^{-1} & s_3-s_3^{-1} & qs_1^{-1}-q^{-1}s_1  \\
0& 0 & s_1-s_1^{-1} & qs_1^{-1}-q^{-1}s_1 & s_3-s_3^{-1}  
\end{pmatrix}.
\end{align} 

Note the symmetry $1\leftrightarrow \bar 1$. In this section, this symmetry is always preserved. Any formula has an analog where all colors are replaced  by that rule.

In the classification of \cite{FJMV}, our matrix is of type $(2,\dots,2,1;1)$D and it corresponds to the product of Fock spaces $\mc F_2^{\otimes n} \otimes \mc F_1 \otimes \mc F_1^{\on{CD}}$ of the  $\mc K_1$ algebra.

We again adopt the notation $Y_{i,\mu}=\bs i_\mu$, $Y_{i,\mu}^{-1}=\bs i^\mu$,  
$Y_{i,\mu}Y_{i,\nu}Y^{-1}_{i,\kappa}=\bs i_{\mu,\nu}^\kappa$, etc.

The roots have the form
\begin{align*}
A_{n,1}&=\bs n_{q,q^{-1}} (\bs{n-1})^1, \qquad 
A_{i,1}=\bs i_{q,q^{-1}} (\bs{i-1})^1(\bs{i+1})^1 
\qquad (i=3,\dots, n-1), \\ 
A_{2,1}&=\bs 2_{q,q^{-1}}\bs 3^1
\bs 1_{s_1}^{s_1^{-1}} \bs {\bar 1}_{s_1}^{s_1^{-1}},
\qquad\quad \   
A_{1,1}=\bs 1_{s_3}^{s_3^{-1}} 
\bs {\bar 1}_{qs_1^{-1}}^{q^{-1}s_1}\bs 2^1.
\end{align*}
The root $A_{\bar 1,1}$ is obtained from $A_{1,1}$  by the symmetry $1\leftrightarrow \bar 1$.

\medskip

In the case we consider, one expects to find $qq$-characters of degree zero which correspond to finite-dimensional modules of quantum affine algebra $U_q\osp(2|2n)$. We describe some of them now. 

We have the $qq$-character $\chi_{1,1}$ corresponding to the vector $(2n+1)$-dimensional representation of $\osp(2|2n)$. (We use the same notation $\chi_{1,1}$ as in the $\gl(2n|1)$ case, we hope it does not create confusion.)

It is described as follows.
\begin{align}\label{one box osp}
\chi_{1,1}=\bs n_1+\bs n^{q^2}(\bs{n-1})_q+(\bs{n-1})^{q^3}(\bs{n-2})_{q^2}+\dots +
\bs 3^{q^{n-1}}\bs 2_{q^{n-2}}+\bs 2^{q^{n}}\bs 1^{q^{n-1}s_1}_{q^{n-1}s_1^{-1}}\bs {\bar 1}^{q^{n-1}s_1}_{q^{n-1}s_1^{-1}}+ \\
+\bs 1^{q^{n-1}s_1}_{q^{n+1}s_1}\bs{\bar 1}_{q^{n-1}s_1^{-1}}^{q^{n+1}s_1^{-1}}
+\bs{ 1}_{q^{n-1}s_1^{-1}}^{q^{n+1}s_1^{-1}} \bs{\bar  1}^{q^{n-1}s_1}_{q^{n+1}s_1}
+\bs{  1}^{q^{n+1}s_1^{-1}}_{q^{n+1}s_1} \bs{\bar  1}^{q^{n+1}s_1^{-1}}_{q^{n+1}s_1}\bs2_{q^n}
+\bs2^{q^{n+2}}\bs 3_{q^{n+1}}+\dots + \bs n^{q^{2n}}.\notag
\end{align}
We also set $\chi_{1,\mu}=\tau_\mu(\chi_{1,1})$.

Note that $q^{2n}$ is the square of the central charge of $\mc K_1$ algebra acting in the product $\mc F_2^{\otimes n} \otimes \mc F_1 \otimes \mc F_1^{\on{CD}}$.

The $qq$-character $\chi_{1,1}$ corresponds to the representation of $\osp(2|2n)$ which is in the standard way denoted by a Young diagram with a single box. Then the $2n+2$ terms of the $qq$-character correspond to a basis of this representation labeled by the semi-standard Young tableaux with alphabet $\{n,n-1,\dots,1,0,\bar 0,1,\dots,\bar n\}$. Thus, one can view this representation as described on Figure \ref{one box osp pic}. 


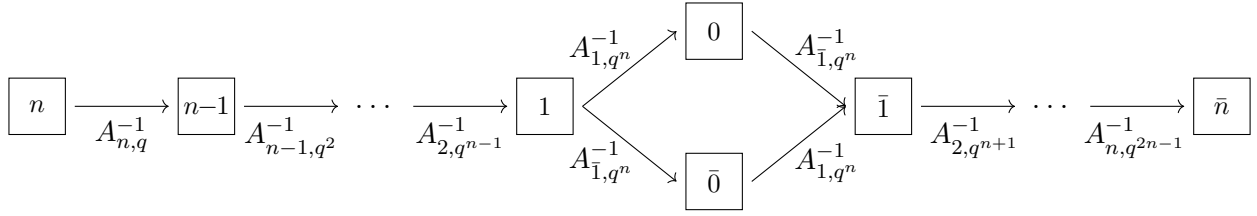
\begin{figure}
\begin{align*}
\begin{tikzpicture}
\node[rectangle,draw,minimum width = 0.75cm, 
    minimum height = 0.75cm] (r) at (-9,0) {};
 \node[rectangle,draw,minimum width = 0.75cm, 
    minimum height = 0.75cm] (r) at (-6.75,0) {};
    \node[rectangle,draw,minimum width = 0.75cm, 
    minimum height = 0.75cm] (r) at (-2.25,0) {};
    \node[rectangle,draw,minimum width = 0.75cm, 
    minimum height = 0.75cm] (r) at (0,1) {};
    \node[rectangle,draw,minimum width = 0.75cm, 
    minimum height = 0.75cm] (r) at (0,-1) {};
    \node[rectangle,draw,minimum width = 0.75cm, 
    minimum height = 0.75cm] (r) at (2.25,0) {};
    \node[rectangle,draw,minimum width = 0.75cm, 
    minimum height = 0.75cm] (r) at (6.75,0) {};
    \node at (-9,0) {\footnotesize$n$};
    \node at (-6.75,0){\footnotesize$n\hspace{-3pt}-\hspace{-3pt}1$};
    \node at (-4.5,0) {$\dots$};
    \node at (4.5,0) {$\dots$};
    \node at (-2.25,0) {\footnotesize$1$};
        \node at (0,1) {\footnotesize$0$};
        \node at (0,-1) {\footnotesize$\bar 0$};
        \node at (2.25,0) {\footnotesize$\bar 1$};
        \node at (6.75,0) {\footnotesize$\bar n$};
\draw[ ->] (-8.5,0)-- node[below]{{\small $A_{n,q}^{-1}$}}(-7.25,0);
\draw[ ->] (-6.25,0)-- node[below]{{\small $A_{n-1,q^2}^{-1}$}}(-5,0);
\draw[ ->] (-4,0) --node[below] {{\small $A_{2,q^{n-1}}^{-1}$}} (-2.75,0);
\node at (-1.5,0.75) {{\small $A_{1,q^n}^{-1}$}} ;
\draw[ ->] (-1.75,0)-- (-0.5,1);
\draw[ ->] (-1.75,0)-- (-0.5,-1);
\node at (-1.5,-0.75) {{\small $A_{\bar 1,q^n}^{-1}$}};
\draw[ ->] (0.5,1)-- (1.75,0);
\node at (1.5, 0.75) {\small $A^{-1}_{\bar 1,q^n}$};
\draw[ ->] (0.5,-1)-- (1.75,0);
\node at (1.5,-0.75) {\small $A^{-1}_{1,q^n}$};
\draw[ ->] (2.75,0)-- node[below]{\hspace{5pt}\small $A^{-1}_{2,q^{n+1}}$}(4,0);
\draw[ ->] (5,0)-- node[below]{\small $A^{-1}_{n,q^{2n-1}}\hspace{2pt}$}(6.25,0);
\end{tikzpicture}
\end{align*}
\caption{The $qq$-character corresponding to the $\osp(2|2n)$ vector representation.}\label{one box osp pic}
\end{figure}

\medskip 

It turns out that the representations 
corresponding to the one column Young diagram with $k$ boxes $(k=1,\dots,n-1)$ of type $\osp(2|2n)$ do not correspond to basic $qq$-characters.

For example, let $n=3$. One would expect to find a $qq$-character with the top term $\bs 2_1$ inside the product of $\chi_{1,q}\chi_{1,q^{-1}}$. However this product is not generic. 
Namely the term labeled by $2$ in $\chi_{1,q}$ is $\bs 3^{q^3}\bs 2_{q^2}$. 
The term labeled by $\bar 1$ in $\chi_{1,q^{-1}}$ is $\bs 1^{q^3s_1^{-1}}_{q^3s_1}\bs{\bar 1}^{q^3s_1^{-1}}_{q^3s_1}\bs 2_{q^2}$. These two terms share $\bs 2_{q^2}$ and therefore they are not mutually generic. In this paper we do not discuss non-basic $qq$-characters.

\medskip

Now we exhibit a $qq$-character $\xi_\mu$ of degree $(0,\dots,0,1,-1)$ which has $2^n$ terms.
We have $\xi_\mu=\tau_\mu(\xi_1)$ and 
$\xi_1$ is defined recursively as follows. 

To describe the recursion, we explicitly show the dependence on $n$ and write $\xi_\mu=\xi_\mu^{(n)}$.
We have  
$$
\xi_1^{(1)}=\bs 1_1 \bs{\bar 1}^{s_1^2}+\bs 1_{q^2s_1^2}\bs{\bar 1}^{q^2},
$$
$$
\xi_1^{(2)}=\bs 1_1 \bs{\bar 1}^{s_1^2}+\bs 1_{q^2s_1^2}\bs 2_{q s_1}\bs{\bar 1}^{q^2}+ 
\bs 1_{q^2} \bs 2^{q^3s_1}\bs{\bar 1}^{q^2s_1^2}+
\bs 1_{q^4s_1^2}\bs{\bar 1}^{q^4},
$$
\begin{align*}
\xi_1^{(3)}=\bs 1_1 \bs{\bar 1}^{s_1^2}+\bs 1_{q^2s_1^2}\bs 2_{q s_1}\bs{\bar 1}^{q^2}+ 
\bs 1_{q^2} \bs 2^{q^3s_1}\bs{\bar 1}^{q^2s_1^2}\bs 3_{q^2s_1}
+
\bs 1_{q^4s_1^2}\bs{\bar 1}^{q^4}\bs 3_{q^2s_1} \\
+\bs 1_{q^2}\bs 3^{q^4s_1} \bs{\bar 1}^{q^2s_1^2}+\bs 1_{q^4s_1^2}\bs 2_{q^3s_1}\bs 3^{q^4s_1}\bs{\bar 1}^{q^4}+ 
\bs{1}_{q^4} \bs 2^{q^5s_1}\bs{\bar 1}^{q^4s_1^2}+
\bs 1_{q^6s_1^2}\bs{\bar 1}^{q^6}.
\end{align*}
The picture of the case $n=3$ is shown in Figure \ref{char gl(2n|1)}. Note that though affine roots are different and the top monomial is different, the structures of the $\xi_{1}$ are the same in the cases of $\gl(2n|1)$ and $\osp(2|2n)$.

In general we set 
\begin{align}\label{init osp}
\xi_\mu^{(n)}=\xi_\mu^{(n),1}+\xi_\mu^{(n),2}, \qquad 
\xi_\mu^{(1),1}=\bs 1_\mu\bs{\bar 1}^{\mu s_1^2}, \quad \xi_\mu^{(1),2}=\bs 1_{\mu q^2 s_1^2}\bs{\bar 1}^{\mu q^2},
\end{align}
and use recursion \eqref{rec}.

\begin{lem}\label{osp eta lem}
The recursion  \eqref{init osp}, \eqref{rec} defines a basic $qq$-character $\xi_1^{(n)}$ with $2^n$ terms corresponding to the deformed Cartan matrix of $\osp(2|2n)$ type, cf. \eqref{osp cartan} of degree $(0,\dots,0,1,-1)$.
\end{lem}
\begin{proof}
The proof is the same as the proof of Lemma \ref{eta lem}.
\end{proof}

We solve the recursion explicitly.

\begin{lem}\label{xi osp explicit lem} We have
\begin{align*}
\xi^{(n)}= &\sum_{\nu\in\{0,1\}^n}\tilde\xi_\nu \\
&\tilde\xi_{\nu_n,\ldots,\nu_1}=\prod_{i=2}^n
Y_{i,q^{i+2\sum_{j=i+1}^n\nu_j+\nu_i-\nu_{i-1}}s_1}^{-\nu_i+\nu_{i-1}}
\cdot
Y_{1,q^{2\sum_{j=1}^n\nu_j}s_1^{2\nu_1}}
\cdot 
Y_{\bar 1,q^{2\sum_{j=1}^n\nu_j}s_1^{2-2\nu_1}}^{-1}\,.
\end{align*}

\end{lem}
\begin{proof}
By Lemma \ref{osp eta lem}, the $qq$-character $\xi^{(n)}$ is given by the same recursion
\eqref{xi-rec}. Thus the lemma is proved in the same way as  Lemma \ref{xi explicit lem}, 
the only difference being the initial condition
$\tilde\xi_0=\bs{1}_1\bar{\bs{1}}^{s_1^2}$, $\tilde\xi_1=\bs{1}_{q^2s_1^2}\bar{\bs{1}}^{q^2}$.  
\end{proof}

By the symmetry $1 \leftrightarrow \bar 1$, using  $\xi_1$ we also obtain a basic $qq$-character with $2^n$ terms of degree $(0,\dots,0,1,-1)$ which we denote $\eta_1$.

\subsection{The case of $\gl(n|m)$}\label{gl(n|m) sec}
We discuss the case of $\gl(n|m)$ for arbitrary $m,n$ and the standard parity.

The Dynkin diagram is shown on Figure \ref{glnm fig}. 
Note that for $m=1$ the present choice of the Dynkin diagram is different from the symmetric one 
used in Subsection \ref{gl(2n|1) sec}.
\begin{figure}
\begin{align*}
\begin{tikzpicture}
\dynkin[root radius=.2cm, edge length=1.3cm, 
labels={n-1,1,0, \overline{1}, \overline{m-1}}]{A}{o.oto.o}
\end{tikzpicture}
\end{align*}
\caption{The $\gl(n|m)$ Dynkin diagram and labeling.}\label{glnm fig}
\end{figure}
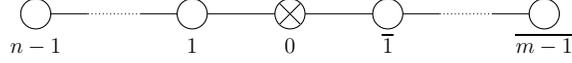

Let again, $s_1,s_2=q,s_3$ satisfy $s_1s_2s_3=1$.
We study the deformed $(n+m-1)\times (n+m-1)$ Cartan matrix whose non-trivial entries are given by:

\begin{itemize}
\item
$c_{ii}=q+q^{-1}$ $(i=1,\dots,n-1)$, $c_{00}=t_3$, $c_{jj}=s_1+s_1^{-1}$ $(j=\bar 1,\dots,\overline{m-1})$,

\item
$c_{01}=t_1$, $c_{0,\bar1}=t_2$,
\item
$c_{i,i+1}=-1$  $(i=1,\dots,n-1)$, $c_{i+1,i}=-1$  $(i=0,\dots,n-1)$, 
\item
$c_{ \overline{j+1},\bar{j}}= -1$  $(j=0,\dots,m-1)$, $c_{ \bar{j},\overline{j+1}}= -1$  $(j=1,\dots,m-1)$.
\end{itemize}

We have $d_i=-t_1t_3$ $(i= 1,\dots, n-1)$, $d_0=t_3$, $d_j=-t_2t_3$ $(j=\bar 1,\dots,\overline{m-1})$.

As an example we write the matrix for the case $\gl(3|4)$:
\begin{align}\label{gl(3,4) cartan}
C^s_{\gl(3|4)}=\begin{pmatrix}
q+q^{-1} & -1 & 0 & 0 & 0 & 0 \\
-1 & q+q^{-1} & -1 & 0 & 0 & 0 \\
0 & s_1-s_1^{-1} & s_3-s_3^{-1} & q-q^{-1} & 0 & 0  \\
0 & 0 & -1 & s_1+s_1^{-1} & -1 & 0 \\
0& 0 & 0 & -1 & s_1+s_1^{-1} & -1 \\
0& 0& 0 & 0 & -1 & s_1+s_1^{-1}
\end{pmatrix}.
\end{align}

In the classification of \cite{FJMV}, our matrix is of type $(2,\dots,2,1,\dots,1)$A and it corresponds to the product of Fock spaces $\mc F_2^{\otimes n} \otimes \mc F_1^{\otimes m}$ of the quantum toroidal $\gl_1$ algebra.

As always, we adopt the notation $Y_{i,\mu}=\bs i_\mu$, $Y_{i,\mu}^{-1}=\bs i^\mu$,  
$Y_{i,\mu}Y_{i,\nu}Y^{-1}_{i,\kappa}=\bs i_{\mu,\nu}^\kappa$, etc.

Then the roots have the form
\begin{align*}
A_{n-1,1}&=(\bs{n-1})_{q,q^{-1}} (\bs{n-2})^1, \qquad A_{i,1}=\bs i_{q,q^{-1}} (\bs{i-1})^1(\bs{i+1})^1 \qquad (i=2,\dots, n-2), \\ A_{1,1}&=\bs 1_{q,q^{-1}}\bs 2^1\bs 0_{s_1}^{s_1^{-1}},\qquad\qquad \qquad   
\textcolor{red}{A_{0,1}}  
=\bs 0_{s_3}^{s_3^{-1}} \bs 1_1\bs {\bar 1}_1, \qquad \qquad \qquad 
A_{\bar 1,1}=\bs{\bar  1}_{s_1,s_1^{-1}}\bs{\bar  2}^1\bs 0_{q}^{q^{-1}},\\
A_{\overline{m-1},1}&=(\bs{\overline {m-1}})_{s_1,s_1^{-1}} (\bs{\overline{m-2}})^1, \quad A_{\bar j,1}=\bs {\bar j}_{s_1,s_1^{-1}} (\bs{\overline{j-1}})^1(\bs{\overline{j+1}})^1 \qquad (j=2,\dots, m-2).
\end{align*}

\medskip

In the case we consider, one expects to find $qq$-characters of degree zero which correspond to finite-dimensional modules of quantum affine algebra $U_q\gl(n|m)$. 

The polynomial $\gl(m|n)$ modules are labeled by 
$(m|n)$ hook partitions. 
The corresponding $qq$-characters in arbitrary parity are described similarly to the corresponding $q$-characters, see Theorem 3.4 of \cite{LM}. We give the construction.

First, we have the $qq$-character $\chi_{1,1}$ corresponding to the vector $(n+m)$-dimensional representation of $\gl(n|m)$. It is described as follows.
\begin{align}
\chi_{1,1}=&(\bs{n-1})_1+(\bs {n-1})^{q^2}(\bs{n-2})_q+(\bs{n-2})^{q^3}(\bs{n-3})_{q^2}+\dots 
\label{chi-glnm}\\
&+
\bs 2^{q^{n-1}}\bs 1_{q^{n-2}}
+\bs 1^{q^n}
\bs 0^{q^{n-1}s_1}_{q^{n-1}s_1^{-1}}
+\bs 0^{q^{n-1}s_1}_{q^{n+1}s_1}\bs{\bar 1}_{q^n}+ \bs{\bar 1}^{q^{n}s_1^2}\bs{\bar 2}_{q^ns_1}+\dots + (\bs{\overline{m-1}})^{q^ns_1^m}.
\nn
\end{align}
We also set $\chi_{1,\mu}=\tau_\mu(\chi_{1,1})$.

Note that $q^{n}s_1^m$ is the central charge of the product $\mc F_2^{\otimes n} \otimes \mc F_1^{\otimes m}$.

The $qq$-character $\chi_{1,1}$ corresponds to the representation of $\gl(n|m)$ which is in the standard way denoted by a Young diagram with a single box. Then the $n+m$ terms of the $qq$-character correspond to a basis of this representation labeled by the semi-standard Young tableaux with the alphabet $\{n,n-1,\dots, 1,\bar 1, \dots,\bar m\}$. One can view this representation as described on Figure \ref{one box gl(n|m) pic}. 

We choose  
an order of the alphabet $\{n,n-1,\dots,1,\bar 1,\dots,\bar n\}$ along the arrows on Figure \ref{one box gl(n|m) pic}: $n\prec n-1\prec n-2\prec \dots\prec 1\prec \bar 1\prec \dots \prec\bar m$.   

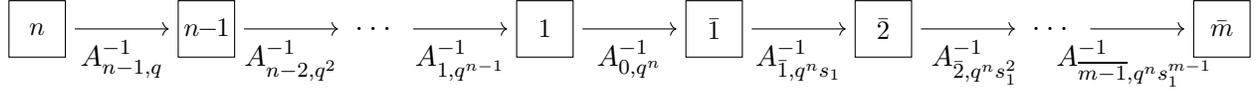
\begin{figure}
\begin{align*}
\begin{tikzpicture}
\node[rectangle,draw,minimum width = 0.75cm, 
    minimum height = 0.75cm] (r) at (-9,0) {};
 \node[rectangle,draw,minimum width = 0.75cm, 
    minimum height = 0.75cm] (r) at (-6.75,0) {};
    \node[rectangle,draw,minimum width = 0.75cm, 
    minimum height = 0.75cm] (r) at (-2.25,0) {};
    \node[rectangle,draw,minimum width = 0.75cm, 
    minimum height = 0.75cm] (r) at (0,0) {};
    \node[rectangle,draw,minimum width = 0.75cm, 
    minimum height = 0.75cm] (r) at (2.25,0) {};
    \node[rectangle,draw,minimum width = 0.75cm, 
    minimum height = 0.75cm] (r) at (6.75,0) {};
    \node at (-9,0) {\footnotesize$n$};
    \node at (-6.75,0){\footnotesize$n\hspace{-3pt}-\hspace{-3pt}1$};
    \node at (-4.5,0) {$\dots$};
    \node at (4.5,0) {$\dots$};
    \node at (-2.25,0) {\footnotesize$1$};
        \node at (0,0) {\footnotesize$\bar 1$};
        \node at (2.25,0) {\footnotesize$\bar 2$};
        \node at (6.75,0) {\footnotesize$\bar m$};
\draw[ ->] (-8.5,0)-- node[below]{{\small $A_{n-1,q}^{-1}$}}(-7.25,0);
\draw[ ->] (-6.25,0)-- node[below]{{\small $A_{n-2,q^2}^{-1}$}}(-5,0);
\draw[ ->] (-4,0) -- node[below]{{\small $A_{1,q^{n-1}}^{-1}$}} (-2.75,0);
\draw[ ->] (-1.75,0)-- node[below]{{\small $A_{0,q^n}^{-1}$}}(-0.5,0);
\draw[ ->] (0.5,0)-- node[below]{\small $A^{-1}_{\bar 1,q^n s_1}$}(1.75,0);
\draw[ ->] (2.75,0)-- node[below]{\hspace{5pt}\small $A^{-1}_{\bar 2,q^{n}s_1^2}$}(4,0);
\draw[ ->] (5,0)-- node[below]{\small $A^{-1}_{\overline{ m-1},q^ns_1^{m-1}}\hspace{2pt}$}(6.25,0);
\end{tikzpicture}
\end{align*}
\caption{The $qq$-character corresponding to the $\gl(n|m)$ vector representation.}\label{one box gl(n|m) pic}
\end{figure}

Next we describe the $qq$-character $\chi_{\la,1}$ corresponding to the polynomial representations of $\gl(n|m)$. Let $\la$ be a hook partition, $\la=(\la_1,\dots,\la_k)$ where $\la_i$ are positive integers such that $\la_i\geq \la_{i+1}$ and $\la_{n+1}<m+1$.
The Young diagram corresponding to $\la$ is represented by boxes centered at $(i,j)$ where $i=1,\dots,k,$ $j=1,\dots,\la_i$. 

The semi-standard Young tableau $T$ of shape $\la$ is a filling of the boxes of $\la$ with elements of the alphabet $\{n,n-1,\dots,1,\bar 1,\dots,\bar m\}$ in non-decreasing order from up down and from the left to right. 
The fermionic fillings $\bar 1,\dots, \bar m $ are allowed to repeat in the same column but not in the same row. 
The bosonic fillings $1,\dots, n$ can repeat in the same row but not in the same column. 
More formally, a semi-standard Young tableau is a map $T$ from the set of boxes $(i,j)$ to the  alphabet 
$\{n,n-1,\dots,1,\bar 1,\dots,\bar m\}$ such that 
$T(i+1,j)\succeq T(i,j)$
with equality allowed only if 
$T(i,j)\succeq \bar 1$, 
and  such that 
$T(i,j+1)\succeq T(i,j)$
with equality allowed only if 
$T(i,j)\preceq 1$.  
Let $T(\la)$ denote the set of all semi-standard Young tableaux of shape $\la$. Now we describe the monomial corresponding to each semi-standard Young tableau $T\in T(\la)$.

We find monomials of $\chi_{\la,1}$ inside of the product $\prod_{i=1}^k\prod_{j=1}^{\la_i}\chi_{1,q^{-2i}s_1^{-2j}}$. Each factor contributes by the monomial corresponding to its filling. 
More precisely, let $M_{i,1}$ be the monomial of $\chi_{1,1}$ corresponding to the filling $i\in\{n,n-1,\dots,1,\bar 1,\dots, \bar m\}$ and $M_{i,\mu}=\tau_\mu(M_{i,1})$.  Then the monomial $M_T$ in $\chi_{k,1}$ corresponding to the semi-standard Young tableau $T$ is
$$
M_T=\prod_{i=1}^k\prod_{j=1}^{\la_i} M_{T((i,j)),q^{-2i}s_1^{-2j}}.
$$

The top term of  $\chi_{k,1}$ corresponds to the minimal filling when the top row contains $n$, the second row $n-1$, the $n$-th row $1$, and then the remainder of the first column contains $\bar 1$, the remainder of the second column contains $\bar 2$, and the remainder of the
$m$-th column contains $\bar m$. 

We have the following lemma. 
\begin{lem}\label{gl(m|n) hook part lem}
For any hook partition $\la$, the sum
$$
\chi_{\la,1}=\sum_{T\in T(\la)} M_T,
$$
is a degree zero basic $qq$-character corresponding to the deformed Cartan matrix of $\gl(n|m)$ type, cf. \eqref{gl(3,4) cartan}.
\end{lem}
\begin{proof}
For $i\in \{n,\dots,1,\bar 1,\dots,\bar m\}$, let 
$i^-$ be the preceding element.

We need to check the decomposition of $\chi_{k,1}$ into elementary blocks of color $i$ where $i\in\{n-1,\dots,1,0,\bar 1,\dots,\overline{m-1}\}$.
For example, let $i\in\{n-1,\dots,1\}$.  The variables $\bs i_a$ or $\bs i^a$ are present only in monomials of the form 
$M_{i^-,b}$
and $M_ {i,b}$. 
If for some box $(l,s)$, the filling is $i^-$: $T((l,s))=i^-$.
The the $i_a$ from the corresponding monomial is not canceled if and only if 
$T'$ defined by 
$$
T'=T, \quad {\rm{except}} \qquad T(l,s)=i
$$
is also a semi-standard Young tableau. In that case $M_T+M_T'$ 
has a factor which is an elementary block of color $i$ of length $2$. 
Thus $\chi_\la$ decomposes into products of elementary blocks
of color $i$ of length $2$ and length $1$.
\end{proof}

We consider the special hook partition $\la^{(0)}=(m,\dots,m)$, where
$m$ is repeated $n$ times. Thus, the Young diagram of $\la^{(0)}$ is the rectangle of size $n\times m$. 

It is known that the corresponding polynomial representation of $\gl(m|n)$ is a Kac module generated from the trivial representation of the even subalgebra $\gl(n)\times \gl(m)$. In particular it has dimension $2^{nm}$. 
The next lemma is the reflection of the fact that the structure of such a
Kac module does not depend on the weight component corresponding to the fermionic root of color $m$. 

\begin{lem}
The $qq$-character $\chi_{\la^{(0)},1}$ is divisible by 
$\bs 0^{q^{-n-1}s_1^{-1}}$. 
In other words the Laurent polynomial 
$\xi_1=\tau_{q^{n+1}s_1^{2m+1}}(\bs 0_{q^{-n-1}s_1^{-1}} \chi_{\la^{(0)},1})$ 
is a basic $qq$-character of degree $(0,\dots,0,1,0,\dots,0)$ corresponding to the deformed Cartan matrix of $\gl(n|m)$ type, cf. \eqref{gl(3,4) cartan}. The $qq$-character 
$\xi_1$ is a sum of $2^{mn}$ monomials with the top term $\bs 0_1$.
\end{lem}
\begin{proof}
The variables $\bs 0_a$ or $\bs 0^a$ are present only in monomials of the form $M_{1,b}$ and $M_{\bar 1,b}$. The key observation is that for any $T\in T(\la^{(0)})$, the filling of the bottom left corner is either $1$ or $\bar 1$: $T((n,1))\in\{1,\bar 1\}$. Thus $M_T$ has either a factor 
$M_{1,q^{-2n}s_1^{-2}}$ or a factor 
$M_{\bar 1,q^{-2n}s_1^{-2}}$. 
Both these monomials contain 
$\bs 0^{q^{-n-1}s_1^{-1}}$. 
There is no cancellation since  
$M_{1,q^{-2n}}$
and 
$M_{\bar 1,q^{-2n-2}s_1^{-2}}$
are never present.
\end{proof}

We also set $\eta_\mu=\bs 0^{\mu}$. So $\eta_\mu$ is a basic $qq$-character of degree $(0,\dots,0,-1,0,\dots,0)$. 

We have $\xi_{q^{-n-1}s_1^{-2m+1}}\eta_{q^{-n-1}s_1}=\chi_{\la^{(0)},1}$.

In other parities one expects to find two basic $qq$-characters such that their product has degree zero and $2^{mn}$ terms. One example of that is discussed in Section \ref{gl(2n|1) sec}.

\section{Extended deformed $W$-algebras}\label{sec:Wex}
In this section we discuss the free field realization of $qq$-characters discussed in the previous section. 
We give a set of commutation relations for them and show that they constitute an extension of the standard deformed
$W$-algebras associated with
$\gl(2n|1)$ or $\mathfrak{osp}(2|2n)$. 

\subsection{Free fields}\label{subsec:boson}
First we prepare some generalities concerning free fields. 
For the sake of concreteness, we focus our attention to the case where 
the Cartan matrix $C=\bigl(C_{i,j}\bigr)_{i,j\in I}$ depends on two parameters $s_1,s_2$,  
and the symmetrized matrix $B=DC$ is invariant under the simultaneous change $s_i\to s_i^{-1}$, $i=1,2$. 
We retain the notation $s_3=(s_1s_2)^{-1}$, $q=s_2$.
For a Laurent polynomial $f=f(s_1,s_2)$ and $r\in \Z\backslash\{0\}$, we write $f^{[r]}(s_1,s_2)=f\bigr(s_1^r,s_2^r)$. 

From now on, we specialize $s_1,s_2$ to non-zero complex numbers and use the same letters to denote them,  
assuming that $s_1^as_2^b=1$, $a,b\in\Z$, holds only if $a=b=0$.
We set
\begin{align*}
s_1=s_3^{-\gamma}\,,\quad s_2=s_3^{-(1-\gamma)}\,,\quad \gamma\not\in\Q\,. 
\end{align*}
We assume also that the limit
\begin{align*}
K=-\lim_{s_3\to 1} \bigl(t_3^{-2}B\bigr)\bigr|_{s_1=s_3^{-\gamma}, s_2=s_3^{-(1-\gamma)}}\,
\end{align*}
exists and is invertible. 

Consider a Heisenberg algebra with generators
$\{Q_{\sss_i}, \sss_{i,r}\mid i\in I, r\in \Z\}$, satisfying the commutation relations
\begin{align*}
&[\sss_{i,r},\sss_{j,r'}]=-\frac{1}{r}\delta_{r+r',0}\bigl(t_3^{[r]}\bigr)^{-2}\bigl(B^{[r]}\bigr)_{i,j}\,,\quad
r,r'\neq 0\,,\\
&[\sss_{i,0},Q_{\sss_j}]=K_{i,j}\,.
\end{align*}
All other commutators are set to $0$. 
Define further
\begin{align*}
&\ssx_{i,r}=-\bigl(t_3^{[r]}\bigr)^2\sum_{k\in I} \bigl({B^{[r]}}^{-1}\bigr)_{i,k}\sss_{k,r}\,,\quad r\neq 0\,,\\
&\ssx_{i,0}=\sum_{k\in I} \bigl(K^{-1}\bigr)_{i,k}\sss_{k,0}\,,
\quad
Q_{\ssx_i}=\sum_{k\in I} \bigl(K^{-1}\bigr)_{i,k}Q_{\sss_k}\,.
\end{align*}
We shall be concerned with the following vertex operators for $i\in I$:  
\begin{align*}
&S_i(z)=e^{Q_{\sss_i}}z^{\sss_{i,0}}:\exp\Bigl(\sum_{r\neq 0}\sss_{i,r}z^{-r}\Bigr):\,,\\
&X_i(z)=e^{Q_{\ssx_i}}z^{\ssx_{i,0}}:\exp\Bigl(\sum_{r\neq 0}\ssx_{i,r}z^{-r}\Bigr):\,,\\
&A_i(z)=:\frac{S_i(s_3^{-1}z)}{S_i(s_3z)}:\,,\\
&Y_i(z)=\begin{cases}
\ds{:\frac{X_i(s_1z)}{X_i(s_1^{-1}z)}:}& \text{for $i$ bosonic},\\
\ds{:\frac{1}{X_i(z)}:}& \text{for $i$ fermionic}\,.\\
\end{cases}
\end{align*}
Here and after we use the standard normal ordering symbol  $:~:$, under which  
$Q_{\sss_i}$, $\sss_{i,-r}$ are placed to the left and $\sss_{i,0}$, $\sss_{i,r}$ to the right, for all $i\in I$ and $r>0$.
We call $S_i(z)$ the screening currents, $A_i(z)$ the root currents and $Y_i(z)$ the $Y$ currents.
We remark that, while the root currents and the bosonic Y currents do not depend on the $Q_{\sss_i}$'s,
the fermionic Y currents do.  

Quite generally, a product of two vertex operators $V_i(z)$, $i=1,2$,  has the form
\begin{align*}
V_1(z)V_2(w)=z^\alpha \varphi_{V_1,V_2}(w/z):V_1(z)V_2(w):\,,
\end{align*}
where $\alpha\in\C$ and $\varphi_{V_1,V_2}(w/z)$ is a formal power series in $w/z$. We 
use the symbol 
\begin{align*}
\cont{V_1(z)}{V_2(w)}=z^\alpha \varphi_{V_1,V_2}(w/z)
\end{align*}
and call it the contraction of $V_1(z)$ and $V_2(w)$. Clearly we have
\begin{align*}
&\cont{V_1(z)}{\bigl(:V_2(w)V_3(w):\bigr)}=\cont{V_1(z)}{V_2(w)}\times \cont{V_1(z)}{V_3(w)}\,,\\
&\cont{\bigl(:V_2(w)V_3(w):\bigr)}{V_1(z)}=\cont{V_2(w)}{V_1(z)}\times \cont{V_3(w)}{V_1(z)}\,,
\end{align*}
for vertex operators $V_i(z)$, $i=1,2,3$.

The contractions of $S_i(z)$ and $Y_j(w)$ converge to rational functions:
\begin{align*}
&\cont{S_i(z)}{Y_i(w)}=\frac{z-s_1w}{z-s_1^{-1}w}=\cont{Y_i(w)}{S_i(z)}\quad \text{for $i$ bosonic}\,,\\
&\cont{S_i(z)}{Y_i(w)}=\frac{1}{z-w}=-\cont{Y_i(w)}{S_i(z)}\quad \text{for $i$ fermionic}\,,\\
&\cont{S_i(z)}{Y_j(w)}=1=\cont{Y_j(w)}{S_i(z)}\,,\quad i\neq j\,.
\end{align*}
These formulas should be understood as appropriate expansions in powers of $w/z$ or $z/w$. 

We note in particular that
\begin{align*}
&\cont{A_i(z)}{Y_j(w)}=1=\cont{Y_j(w)}{A_i(z)}\,,\quad i\neq j\,,
\end{align*}
and  that the contractions
\begin{align*}
\cont{A_i(z)}{Y_i(w)}\,,\quad\cont{Y_i(w)}{A_i(z)}\,,\quad \cont{A_i(z)}{A_j(w)} 
\end{align*}
are all rational, the first two being the same rational function. 

In contrast, the contractions of the screening currents
\begin{align*}
\cont{S_i(z)}{S_j(w)}=z^{K_{i,j}}\exp\Bigl(-\sum_{r>0}\frac{1}{r}\frac{B^{[r]}_{i,j}}{(t_3^{[r]})^2}\frac{w^r}{z^r}\Bigr)
\end{align*}
are non-rational because of the denominator $t^{[r]}_3=s_3^r-s_3^{-r}$ (note that $B^{[r]}_{i,j}$ is divisible by $t_3^{[r]}$).
If $|s_3|>1$, 
for example, then these contractions
are meromorphic functions written in terms of infinite products of the form $(z;s_3^{-2})_\infty=\prod_{j=0}^\infty(1-s_3^{-2j}z)$. 
Similarly the contractions between $Y_i(z)$'s
are non-rational (and are slightly more complicated than those of $S_i(z)$'s). 

For a generic monomial  $m=\prod_{i\in I}\prod_{a\in\C}Y^{n_{i,a}}_{i,a}$, $n_{i,a}=\pm1$, we set 
\begin{align}
\quad m(z)=:\prod_{i\in I}\prod_{a\in\C}Y_{i}(a z)^{n_{i,a}}:\,.
\label{m(z)}
\end{align}
The contractions of $S_i(w)$ and $m(z)$  are rational functions with only simple poles, so that the commutator is a finite sum 
\begin{align*}
[S_i(w),m(z)]=\sum_{b}d_{i,{m},b}w^{-1}\delta\Bigl(\frac{bz}{w}\Bigr) \,:S_i(bz)m(z): \,,
\end{align*}
where $d_{i,{m},b}\in\C$  and $\delta(z)=\sum_{n\in \Z}z^n$ stands for the delta function. 
Given a basic $qq$-character $\chi=\sum_m m$ 
we define its bosonization to be the current
\begin{align}
V_\chi(z)=\sum_{m} c_{m} m(z)\,,\quad c_m\in\C\,,\label{Vchi}
\end{align}
such that for all $i\in I$ we have 
\begin{align*}
\sum_m \sum_b c_m d_{i,m,b}:S_i(bz) m(z):=0\,.
\end{align*}
We write this relation symbolically as 
\begin{align*}
[\int S_i(w)dw, V_\chi(z)]=0
\end{align*}
and say that $V_{\chi}(z)$ formally commutes with the screening operators $\int S_i(w)dw$. 
In all our examples below, the coefficients $c_m$ are determined uniquely up to an overall scalar multiple 
by the formal commutativity with screening operators.

\subsection{Bosonized $qq$-characters for $\gl(2n|1)$}\label{boson gl(2n|1) sec}

Consider now the case $\gl(2n|1)$ given by the Cartan matrix indexed by $I=\{n,\ldots,1,\bar1,\ldots,\bar n\}$;
see subsection \ref{gl(2n|1) sec}, \eqref{gl(2n|1) cartan}.
The non-zero entries of  the matrix $K$ are 
\begin{itemize}
\item $K_{i,i}=-2\gamma$ for $i\neq 1,\bar1$ and  $K_{1,1}=K_{\bar1,\bar1}=-1$,
\item $K_{i,i\pm1}=K_{\bar i,\overline{i\pm1}}=\gamma$  and $K_{1,\bar1}=K_{\bar1,1}=1-\gamma$,
\end{itemize}
and $\det K=\gamma^{2n-1}\bigl(\gamma+2n(1-\gamma)\bigr)\neq 0$. 
For example, for $n=3$ we have 
\begin{align}\label{gl(2n|1) K}
K_{\gl(6|1)}=\begin{pmatrix}
-2\gamma & \gamma & 0 & 0 & 0 & 0 \\
\gamma & -2\gamma & \gamma & 0 & 0 & 0 \\
0 & \gamma & -1 & 1-\gamma & 0 & 0  \\
0 & 0 & 1-\gamma &-1 & \gamma & 0 \\
0& 0 & 0 & \gamma & -2\gamma & \gamma \\
0& 0& 0 & 0 & \gamma & -2\gamma
\end{pmatrix}.
\end{align} 

First consider the degree zero $qq$-character $\chi_{1,1}=\sum_{i\in I}M_{i,1}$ given by \eqref{one box form} and shown on Figure \ref{one box gl(n|m) pic}.
Using notation \eqref{m(z)}, \eqref{Vchi} we have
\begin{align*}
V_{\chi_{1,1}}(z)=\sum_{i=1}^n q^{2i-1}s_1 M_{i,1}(z)+\frac{s_1-s_1^{-1}}{q-q^{-1}} M_{0,1}(z)+\sum_{i=1}^n q^{-2i+1}s_1^{-1} M_{\bar i,1}(z)\,.
\end{align*}
For general $k\ge1$, $\chi_{k,1}$ given in Lemma \ref{gl 2n one column lem}, is bosonized as
\begin{align}
V_{\chi_{k,1}}(z)=\sum_{i_1,\ldots,i_k\in\{n,\ldots,1,0,\bar1,\ldots,\bar n\}\atop i_1\preceq\ldots\preceq i_k}
c^{\chi}_{i_1,\ldots,i_k}:M_{i_1,q^{k-1}}(z)M_{i_2,q^{k-3}}(z)\cdots M_{i_k,q^{-k+1}}(z):\,.
\label{Vchi-gl}
\end{align}
As before, in the sum, the equality $i_s=i_{s+1}$ is allowed only if $i_s=i_{s+1}=0$. 
To describe the coefficients $c^{\chi}_{i_1,\ldots,i_k}$, let $r$ be the number of times $0$ appears in 
$i_1,\ldots,i_k$. Then 
\begin{align}
c^{\chi}_{i_1,\ldots,i_k}=\prod_{i_p\in \{n,\ldots,1\}}q^{2i_p-1}s_1 \prod_{\overline{i_p}\in\{\bar 1,\ldots,\bar n\}}q^{-2i_p+1}s_1^{-1}
\prod_{j=1}^r \frac{q^{-j+1}s_1-q^{j-1}s_1^{-1}}{q^j-q^{-j}}\,.
\label{cchi}
\end{align}

The currents $V_{\chi_{k,1}}(z)$, $k\ge1$,  
are the generators of the $W$-algebra associated with $\gl(2n|1)$. 
Their commutation relations have been studied in \cite{K1}.

Next consider the $qq$-character $\xi^{(n)}=\xi^{(n)}_1$ of degree $(0,\ldots,0,1,-1,0,\ldots,0)$, given in Lemma \ref{xi explicit lem}.

Note that 
\begin{align*}
\tilde\xi_{0,\ldots,0}=\bs{1}_1\bar{\bs{1}}^{s_1}\,,
&\quad 
\tilde\xi_{1,0,\ldots,0}=\bs{1}_{q^2}\bar{\bs{1}}^{q^2s_1}\bs{n}^{q^{n+1}s_1}\,.
\end{align*}
Each term forms an elementary block of color $i$
\begin{align*}
\tilde\xi_{\nu_n,\ldots,\nu_1}(1+A^{-1}_{i,q^{2(\nu_n+\cdots+\nu_{i+1})+i}s_1})
\end{align*}
if and only if $i>1$ and $(\nu_i,\nu_{i-1})=(0,1)$, or $i=1$ and $\nu_1=0$.

The bosonization of $\xi^{(n)}$ is then 
\begin{align}\label{Vxi}
&V_{\xi^{(n)}}(z)=\sum_{\nu_n\ldots,\nu_1\in\{0,1\}}c^\xi_{\nu_n,\ldots,\nu_1}\tilde\xi_{\nu_n,\ldots,\nu_1}(z)\,,\\
&c^\xi_{\nu_n,\ldots,\nu_1}=(-1)^{\sum_{i=1}^n\nu_i}q^{n(n-1)/2-2\sum_{i=1}^n(i-1)\nu_i}\,.\nn
\end{align}
Interchanging $\bs{i}$ with $\bar{\bs{i}}$, we obtain the bosonization $V_{\eta}(z)$ of $\eta=\eta^{(n)}_1$. 

\begin{rem}
For bosonic nodes $i$ one can determine the dual screening current $S_i^-(z)$ by the relation
$$
A_i(z)=:\frac{S_i^-(s_1^{-1}z)}{S_i^-(s_1z)}:\,.
$$
One can show that the bosonization $V_\chi(z)$ of any basic $qq$-character $\chi$ formally commutes with the dual screening operator $\int S_i^-(w)dw$ as well: 
\begin{align*}
[\int S_i^-(w)dw, V_\chi(z)]=0.
\end{align*}
\end{rem}

\subsection{Rationalization}\label{subsec:rational}

The bosonized $qq$-characters $V_\chi(z)$ given above are expressed in terms of $l$ bosons, where $l$ is the rank of the Cartan matrix. 
Their contractions are rather involved due to the non-rational nature of those of the $Y$ currents
(see Remark \ref{remV} below).
In order to study the relations among $V_\chi(z)$'s, it is convenient to introduce modified currents of the form
$V^{\mathrm{rat}}_\chi(z)=:W_\chi(z)V_\chi(z):$, 
where $W_\chi(z)$'s are vertex operators  in auxiliary bosons which commute with the original $l$ bosons,
such that the contractions of $V^{\mathrm{rat}}_\chi(z)$ are rational. 
Such a procedure is not unique, and the choice of $W_\chi(z)$ can depend on $\chi$. 
We call $V^{\mathrm{rat}}_\chi(z)$  a rationalization of $V_\chi(z)$.  
By construction,  it formally commutes with all screening operators. 

Let us discuss rationalizations of $qq$-characters for $\gl(2n|1)$.  
To this aim, consider an extended  Heisenberg algebra generated by $\{Q_{\sss_i},\sss_{i,r}\mid i\in I, r\in\Z\}$ together with 
two sets of generators $\{Q_{\lambda},Q_{\bar\lambda}, \lambda_{r},\bar\lambda_r\mid r\in \Z\}$. 
The commutation relations for the latter are given by assigning contractions of the vertex operators
\begin{align*}
&\Lambda(z)=e^{Q_\lambda}z^{\lambda_0}(s_1z)^{-\bar\lambda_0}:\exp\Bigl(\sum_{r\neq0}\lambda_{r}z^{-r}\Bigr)\,,\\
&\bar\Lambda(z)=e^{Q_{\bar\lambda}}z^{\bar\lambda_0}(s_1z)^{-\lambda_0}:\exp\Bigl(\sum_{r\neq0}\bar\lambda_{r}z^{-r}\Bigr)\,.
\end{align*}
We demand the following:
\begin{enumerate}
\item There exist vertex operators  $\Lambda'(z),\bar\Lambda'(z)$ commuting with all $S_i(w)$'s, such that 
\begin{align*}
\Lambda(z)=:\Lambda'(z)\frac{Y_1(z)}{Y_{\bar1}(s_1z)}:\,,\quad
\bar\Lambda(z)=:\bar\Lambda'(z)\frac{Y_{\bar1}(z)}{Y_1(s_1z)}:\,.
\end{align*}
\item We have the contractions
\begin{align}\label{LaLa}
&\cont{\Lambda(z)}{\Lambda(w)}=\frac{z-w}{z-s_3^2w}\,,\quad \cont{\bar\Lambda(z)}{\bar\Lambda(w)}=\frac{z-w}{z-s_3^{-2}w}\,,
\\
&\cont{\Lambda(z)}{\bar\Lambda(w)}=\cont{\bar\Lambda(z)}{\Lambda(w)}=1\,.\nn
\end{align}
\end{enumerate}
The requirement (i) means that 
\begin{align}\label{LaS}&
\cont{\Lambda(z)}{S_i(w)}=-\cont{S_i(w)}{\Lambda(z)}=
\begin{cases}
\frac{1}{z-w} & \text{for $i=1$},\\
s_1z-w  & \text{for $i=\bar1$},\\
\end{cases}
\\
&\cont{\bar\Lambda(z)}{S_i(w)}=-\cont{S_i(w)}{\bar\Lambda(z)}=
\begin{cases}
s_1z-w  & \text{for $i=1$},\\
\frac{1}{z-w} & \text{for $i=\bar 1$},\\
\end{cases}
\nn\\
&\cont{\Lambda(z)}{S_i(w)}=\cont{\bar\Lambda(z)}{S_i(w)}=\cont{S_i(w)}{\Lambda(z)}=\cont{S_i(w)}{\bar\Lambda(z)}=1 
\quad\text{for $i\neq 1,\bar1$.}\nn
\end{align}

Modifying \eqref{Vchi}, we define a rationalization of $V_\xi(z)$ by
\begin{align}
&E^{(n)}(z)=:\Lambda'(z)V_{\xi^{(n)}}(z):
=\sum_{\nu\in\{0,1\}^n}c^{\xi}_{\nu}E^{(n)}_{\nu}(z)\,,
\label{E(z)}
\\
&E^{(n)}_{\nu}(z)=:\Lambda'(z)\xi_{\nu}(z):\,.\nn
\end{align}
Similarly we define $F(z)$ by interchanging $\xi$ with  $\eta=\eta^{(n)}_1$, 
$\Lambda(z)$ with $\bar\Lambda(z)$, and $A_i(z)$ with $A_{\bar i}(z)$. 

The Cartan matrix for $\gl(2n|1)$ is naturally a submatrix of that for  $\gl(2n+2|1)$, so 
the root currents $A_i(z)$ for $\gl(2n|1)$ can be viewed as those for $\gl(2n+2|1)$.
This is not the case with the $Y$ currents because the inverse of the Cartan matrices do not have
such a submatrix structure. 
This means that $\xi^{(n)}(z)$, $\Lambda'(z)$ for different $n$ cannot be identified.  
Nevertheless  
the contractions of the currents $E^{(n)}_{\nu}(z)$, $F^{(n)}_{\nu}(z)$
are determined completely by \eqref{LaLa} and \eqref{LaS}, and hence they stabilize in $n$; for example 
\begin{align*}
\cont{E^{(n)}_{0,\mu}(z)}{E^{(n)}_{0,\nu}(z)}
=\cont{E^{(n-1)}_{\mu}(z)}{E^{(n-1)}_{\nu}(z)}\,
\quad
\text{for $\mu,\nu\in\{0,1\}^{n-1}$}.
\end{align*}
For that reason we shall drop the superfix $(n)$ from the notation. 

In particular, $E_{0,\ldots,0}(z)=\Lambda(z)$, and \footnote{The second line of \eqref{E100} is valid for $n\ge2$}
\begin{align}
E_{1,0,\ldots,0}(z)&=:\Lambda(z)\prod_{k=1}^nA_k^{-1}(q^ks_1z):\,\label{E100}\\
    &=:\Lambda'(z)Y_1(q^2z)Y_{\bar1}(q^2z)^{-1}Y_n(q^{n+1}s_1z)^{-1}:\,.\nn
\end{align}
Writing 
\begin{align*}
&E(z)=E^0(z)+E^1(z)\,,\\
&E^i(z)
=\sum_{\nu\in\{0,1\}^{n-1}}c^\xi_{i,\nu}E_{i,\nu}(z)\,,
\end{align*}
we see that $E^0(z)$ is a rationalization of $V_{\xi^{(n-1)}_1}(z)$ for $\gl(2n-2|1)$. 
Moreover, since $Y_n(z)$ commutes with $A_i(w)$ with $i=1,\ldots,n-1$, the second equality of \eqref{E100} shows that  $E^1(z)$
is a rationalization of $V_{\xi^{(n-1)}_{q^2}}(z)$ where  $:\Lambda'(z)Y_n(q^{n+1}s_1z)^{-1}:$ 
is used in place of $\Lambda'(q^2z)$.

\begin{rem}
Since $\tilde\xi_{1,\ldots,1}=\tau_{q^{2n}s_1}\bigl(\tilde\eta_{0,\ldots,0}\bigr)^{-1}$, 
one extra boson $\Lambda'(z)$ is enough for the purpose of just rationalizing $V_\xi(z)$ and $V_\eta(z)$.
However we prefer to use two bosons because formulas are simpler and more symmetric.
\end{rem}

The following lemma will be used to study the relations between the currents $E(z)$: 

\begin{lem}\label{lem:E0E1}
For all $\mu,\nu\in\{0,1\}^{n-1}$ we have
\begin{align}
&\cont{E_{1,\mu}(z)}{E_{1,\nu}(w)}=\cont{E_{0,\mu}(z)}{E_{0,\nu}(w)}\,,
\label{E1E1}\\
&\cont{E_{0,\mu}(z)}{E_{1,\nu}(w)}=
\frac{z-s_3^{-2}w}{z-s_3^2w}\frac{z-s_3^2q^2w}{z-q^2w}\,
\cont{E_{0,\mu}(z)}{E_{0,\nu}(q^2w)}\,,
\label{E0E1}\\
&\cont{E_{1,\nu}(z)}{E_{0,\mu}(w)}=
\frac{q^2s_3^{-2}z-w}{q^2z-w}\,
\cont{E_{0,\nu}(q^2z)}{E_{0,\mu}(w)}\,.
\label{E1E0}
\end{align}
\end{lem}
\begin{proof}
From the foregoing discussions, it is enough to prove \eqref{E1E1} for $\mu=\nu=0$. This can be checked without difficulty. 

Consider \eqref{E0E1}. From the definition, we have
\begin{align*}
:\frac{E_{1,\nu}(w)} {E_{0,\nu}(q^2w)}:
=:\frac{\Lambda'(w)}{\Lambda'(q^2w)}Y_n(q^{n+1}s_1w)^{-1}:
=:\frac{E_{1,0,\ldots,0}(w)} {\Lambda(q^2w)}:\,.
\end{align*} 
Since $E_{0,\mu}(z)$ can be written as $:\Lambda(z)P:$ with $P$ is a polynomial of $A_{i}(a z)^{-1}$ with $i\neq n$, $a\in\C$, we have
\begin{align*}
\cont{E_{0,\mu}(z)}{:\frac{E_{1,0,\ldots,0}(w)} {\Lambda(q^2w)}:}
=\cont{\Lambda(z)}{:\frac{E_{1,0,\ldots,0}(w)} {\Lambda(q^2w)}:}\,.
\end{align*} 
This implies that 
\begin{align*}
\cont{E_{0,\mu}(z)}{E_{1,\nu}(w)}
&=\cont{E_{0,\mu}(z)}{E_{0,\nu}(q^2w)}\cdot \cont{E_{0,\mu}(z)}
{:\frac{E_{1,\nu}(w)}{E_{0,\nu}(q^2w)}:} 
\\
&=\cont{E_{0,\mu}(z)}{E_{0,\nu}(q^2w)}\cdot 
\cont{\Lambda(z)}{:\frac{E_{1,0,\ldots,0}(w)}{\Lambda(q^2w)}:}\,.
\end{align*}
Noting \eqref{E100}
we obtain
\begin{align*}
\cont{\Lambda(z)}{:\frac{E_{1,0,\ldots,0}(w)}{\Lambda(q^2w)}:}
&=\cont{\Lambda(z)}{A_1^{-1}(qs_1w)}
\cdot
\cont{\Lambda(z)}{:\frac{\Lambda(w)}{\Lambda(q^2w)}:}
\\
&=\frac{z-s_3^{-2}w}{z-w}\cdot \frac{z-w}{z-s_3^2w}\frac{z-q^2s_3^2w}{z-q^2w}\,,\\
\end{align*}
which proves \eqref{E0E1}. Equation \eqref{E1E0} can be derived similarly.
\end{proof}

\begin{thm}\label{prop:EE}
The product $E(z)E(w)$ (resp.  $F(z)F(w)$)
is regular except for a simple pole at $z=s_3^2w$ (resp.  $z=s_3^{-2}w$). 
We have the quadratic relations
\begin{align}
&(z-s_3^2w)E(z)E(w)+(w-s_3^2z)E(w)E(z)=0\,,\label{EErel}\\
&(z-s_3^{-2}w)F(z)F(w)+(w-s_3^{-2}z)F(w)F(z)=0\,.\label{FFrel}
\end{align}
\end{thm}
\begin{proof}
We use induction on $n$. 
The statement is easy to check in the case $n=1$. 
Suppose it is true for $n-1$. 
By the induction hypothesis, $(z-s_3^2w)E^0(z)E^0(w)$, $(z-s_3^2w)E^1(z)E^1(w)$ 
are both regular and skew symmetric in $z,w$.
Furthermore, from Lemma \ref{lem:E0E1} we find
\begin{align*}
&(z-s_3^2w)E^0(z)E^1(w)=(z-s_3^{-2}w)\frac{z-s_3^2q^2w}{z-q^2w}
\sum_{\mu,\nu}\cont{E_{0,\mu}(z)}{E_{0,\nu}(q^2w)}\\
&\quad\times
:E_{0,\mu}(z)E_{0,\nu}(q^2w)\Lambda'(w)\Lambda'(q^2w)^{-1}
Y_n(q^{n+1}s_1w)^{-1}:\,,\\
&(z-s_3^2w)E^1(z)E^0(w)=(s_3^{-2}z-w)\frac{q^2z-s_3^2w}{q^2z-w}
\sum_{\mu,\nu}\cont{E_{0,\nu}(q^2z)}{E_{0,\mu}(w)}\\
&\quad\times
:E_{0,\nu}(q^2z)E_{0,\mu}(w)\Lambda'(z)\Lambda'(q^2z)^{-1}
Y_n(q^{n+1}s_1z)^{-1}:\,.
\end{align*}
Due to skew symmetry of $(z-s_3^2w)E^0(z)E^0(w)$, 
\begin{align*}
\frac{z-s_3^2q^2w}{z-q^2w} E^0(z)E^0(q^2w)=
\frac{z-s_3^2q^2w}{z-q^2w}\sum_{\mu,\nu}\cont{E_{0,\mu}(z)}{E_{0,\nu}(q^2w)}:E_{0,\mu}(z)E_{0,\nu}(q^2w):
\end{align*}
is regular, hence so is the right hand side of the first equality. 
Similarly the second is regular. 
Therefore the product $(z-s_3^2w)E(z)E(w)$ has no poles. Combining the relations above, we find that  
\eqref{EErel} holds.

The assertion about $F(z)$ is shown similarly.
\end{proof}

Next we consider the relations between $E(z)$ and $F(z)$. 

\begin{lem}\label{lem:EFcont}
For $\mu=(\mu_n,\ldots,\mu_1)\in\{0,1\}^n$, we set $|\mu|=\sum_{i=1}^n\mu_i$.
Then for all for $\mu,\nu\in\{0,1\}^n$ we have 
\begin{align*}
\cont{E_{\mu}(z)}{F_\nu(w)}=
\prod_{i=1}^{|\mu|}\frac{q^{2i-2-2|\nu|}s_1^{-1}z-w}{q^{2i-2|\nu|}s_1z-w}
\prod_{j=1}^{|\nu|}\frac{z-q^{2j-2}s_1^{-1}w}{z-q^{2j}s_1w}\,.
\end{align*}
\end{lem}
\begin{proof}
The currents $E_\mu(z)$, $F_\nu(z)$  can be written as 
\begin{align*}
E_{\mu}(z)=:\Lambda(z)\prod_{i=1}^{|\mu|}A_1(q^{2i-1}s_1z)P:\,,
\quad
F_{\nu}(z)=:\bar\Lambda(z)\prod_{j=1}^{|\nu|}A_{\bar1}(q^{2j-1}s_1z)Q:\,,
\end{align*}
where $P$ (resp. $Q$) is a polynomial in the currents $A^{-1}_i(az)$ (resp. $A^{-1}_{\bar i}(az)$), 
$2\le i \le n$. They do not participate in the contraction because $\cont{P}{\bar\Lambda}$, 
$\cont{\Lambda}{Q}$ and $\cont{P}{Q}$ are all $1$. 

The rest is a direct calculation using 
\begin{align*}
&\cont{A_i(z)}{Y_i(w)}=\cont{Y_i(w)}{A_i(z)}=\frac{s_3z-w}{s_3^{-1}z-w}\quad \text{for $i=1,\bar{1}$}\,,\\
&\cont{A_1(z)}{A_{\bar 1}(w)}=\frac{z-q^2s_1w}{z-s_1w}\frac{z-q^{-2}s_1^{-1}w}{z-s_1^{-1}w}\,.
\end{align*}
\end{proof}

\begin{thm}\label{[EF] lem}
We have the relation 
\begin{align*}
[E(z),F(w)]=\sum_{k=0}^{n-1}s_3^{-1}a_{n,k}\delta(q^{2n-2k}s_1 z/w)T_k(q^{n-k}s_1 z)
+\sum_{k=0}^{n-1}s_3a_{n,k}\delta(q^{2n-2k}s_1 w/z)\bar{T}_{k}(q^{n-k}s_1 w)\,,
\end{align*}
where  $T_k(z),\bar{T}_k(z)$ are rationalizations of $V_{\chi_{k,1}}(z)$, $V_{\chi_{\bar k,1}}(z)$:
\footnote{For $k=0$, we set $V_{\chi_{0,1}}(z)=V_{\chi_{\bar 0, 1}}(z)=1$.}
\begin{align*}
&T_k(z)= :\Lambda'(q^{-n+k}s_1^{-1}z)\bar\Lambda'(q^{n-k}z) V_{\chi_{k,1}}(z)
:\,,\\
&\bar{T}_k(z)= :\Lambda'(q^{n-k}z)\bar\Lambda'(q^{-n+k}s_1^{-1}z) V_{\chi_{\bar k, 1}}(z)
:\,,
\end{align*}
and $a_{n,k}=s_3^{n-1}\prod_{j=1}^{n-k}(q^{-j}s_1^{-1}-q^js_1)/\prod_{j=1}^{n-k-1}(q^j-q^{-j})$.
\end{thm}
\begin{proof}
By Lemma \ref{lem:EFcont}, $E_\mu(z)F_\nu(w)$ has only simple poles located at 
\begin{align*}
&z=q^{2i}s_1 w,\quad 1\le i\le |\nu|-|\mu|\quad \text{for $|\mu|<|\nu|$}\,,\\
&z=q^{-2i}s_1^{-1} w,\quad 1\le i\le |\mu|-|\nu|\quad \text{for $|\mu|>|\nu|$}\,.
\end{align*}
The products $E(z)F(w)$ and $F(w)E(z)$ coincide as rational functions. 
Therefore their commutator has the stated form with some currents $T_k(z),\bar{T}_k(z)$. 

We fix $k\in \{0,\ldots,n-1\}$ and compute $T_k(z)$. Lemma \ref{lem:EFcont} implies that, 
for given $\mu,\nu$, $E_\mu(z)F_\nu(w)$ has a pole at $z=q^{-2n+2k}s_1^{-1} w$ only if 
\begin{align}
(n-|\mu|)+|\nu|\le k\,.\label{munu}
\end{align}
Let $r=n-|\mu|$, $s=k-|\nu|\ge r$.
To a pair $(\mu,\nu)$ satisfying \eqref{munu}, we associate 
$i_1,\ldots,i_k \in \{n,\ldots,1,0,\bar1,\ldots,\bar n\}$, $i_1\preceq\ldots\preceq i_k$, 
by the following rule: 
\begin{align*}
&\{i\mid \mu_i=0\}=\{i_1,\ldots,i_r\}\,,\quad n\preceq i_1\prec\ldots\prec i_r\preceq 1\,,\\
&\{\bar i\mid \nu_i=1\}=\{i_{s+1},\ldots,i_k\}\,,\quad \bar 1\preceq i_{s+1}\prec\ldots\prec i_k\preceq \bar n\,,\\
&\text{$0$ appears $s-r$ times in $\{i_1,\ldots,i_k\}$}.
\end{align*}
This gives a bijection between the set of $(\mu,\nu)$ satisfying \eqref{munu} and the set of
tableaux of a column with $k$ boxes. 
We claim that, under this bijection, 
$:E_\mu(q^{-n+k}s_1^{-1}z)F_\nu(q^{n-k}z):$ corresponds to $:M_{i_1,q^{k-1}}(z)\cdots M_{i_k,q^{-k+1}}(z):$. 

When $\mu=(0,\ldots,0,1,\ldots,1)$ and $\nu=(0,\ldots,0)$, 
we have 
\begin{align*}
:E_{0,\ldots,0,1,\ldots,1}(q^{-n+k}s_1^{-1}z)F_{0,\ldots,0}(q^{n-k}z):=
:\Lambda'(q^{-n+k}s_1^{-1}z)\bar\Lambda'(q^{n-k}z)Y_{n-k+1}(z):\,,
\end{align*}
which corresponds to the top term of $V_{\chi_{k,1}}(z)$. 

In general, suppose that 
$E_{\mu'}(q^{-n+k}s_1^{-1}z)=:E_\mu(q^{-n+k}s_1^{-1}z)A^{-1}_{i}(az):$. Then 
$\mu_i=0$, $\mu'_i=1$ with some $i=i_p$, $1\le p\le r$, and 
$a=q^{-n+k}s_1^{-1}\cdot q^{2(n-i-p+1)+i}s_1$. 
One can check that $M_{p-1,q^{k-2p+1}}(z)=:M_{p,q^{k-2p+1}}(z)A^{-1}_i(az):$. 
Likewise, if $F_{\nu'}(q^{-n+k}z)=:F_\nu(q^{-n+k}z)A^{-1}_{\bar i}(bz):$, then 
$\nu_i=0$, $\nu'_i=1$ with some $i=i_p$, $s<p\le k$, and 
$M_{p-1,q^{k-2p+1}}(z)=:M_{p,q^{k-2p+1}}(z)A^{-1}_{\bar i}(bz):$. 

It follows that $T_k(z)=:\Lambda'(q^{-n+k}s_1^{-1}z)\bar\Lambda'(q^{n-k}z) V_k(z):$
where $V_k(z)$ is a linear combination of terms occurring in $V_{\chi_{k,1}}(z)$. 
Since $T_k(z)$ formally commutes with all screening operators, $V_k(z)$ coincides with
 $V_{\chi_{k,1}}(z)$ up to an overall scalar multiple. 

The calculation of $\bar{T}_k(z)$ is entirely similar.
\end{proof}

In order to discuss commutation relations between $T_k(z)$'s and $E(z)$, $F(z)$, 
it is more convenient to change the rationalizations of $V_{\chi_{k,1}}$. 
We focus to the case $k=1$ and introduce 
\begin{align}
T(z)=:W'(z)V_{\chi_{1,1}}(z):=\sum_{i\in I}T_i(z)\,,\label{Tmodify}
\end{align}
such that 
\begin{align*}
&\cont{T_n(z)}{\Lambda(w)}=\cont{T_n(z)}{\bar\Lambda(w)}=1\,,
\quad \cont{\Lambda(w)}{T_n(z)}=\cont{\bar\Lambda(w)}{T_n(z)}=1\,.
\end{align*}

\begin{lem}\label{lem:TE} 
For $\nu=(\nu_n,\ldots,\nu_1)\in\{0,1\}^n$ and $1\le i\le n$,  we have
\begin{align}
&\cont{T_i(z)}{E_{\nu}(w)}=
\begin{cases}
1 & \text{if $\nu_i=0$},\\
p(q^{-n+2\sum_{j=i+1}^n\nu_j+2i}s_1w/z) &\text{if $\nu_i=1$}\,,
\end{cases}
\label{TE1}
\\
&\cont{T_{0}(z)}{E_\nu(w)}=s_3^{-2}\frac{z-q^{-n+2\sum_{j=1}^m\nu_j-1}s_1^{-1}w}
{z-q^{-n+2\sum_{j=1}^m\nu_j-1}s_1w}\,,
\label{TE2}\\
&\cont{T_{\bar i}(z)}{E_\nu(w)}=1\,,
\label{TE3}
\end{align}
where
\begin{align*}
p(z)=\frac{1-q^{-1}s_1^{-2}x}{1-q^{-1}x}\frac{1-q s_1^2x}{1-qx}\,.
\end{align*}
\qed
\end{lem}

In particular, $T(z)E(w)$ has at most simple poles at $z=q^{-n+2j+1}s_1w$, $j=0,\ldots,n$.
\begin{proof}
It is easy to check that 
\begin{align*}
\cont{T_i(z)}{E_{0,\ldots,0}(w)}=1\,\quad \text{for $1\le i\le n$}.
\end{align*}
Starting from this and using 
\begin{align*}
&\cont{T_i(z)}{A_j^{-1}(w)}=1 \quad \text{for $j\neq i,i+1$},\quad
\cont{T_i(z)}{A_i^{-1}(w)}=p(q^{-n+1}w/z)\,,\\
&\cont{T_i(z)}{A_i^{-1}(w)A_{i+1}^{-1}(qw)}=1\,,
\end{align*}
we obtain \eqref{TE1}. The relations \eqref{TE2}, \eqref{TE3} are derived from \eqref{TE1}, noting 
$T_0(z)=:T_1(z)A^{-1}_{1}(q^nz):$ and $T_{\bar 1}(z)=:T_0(z)A^{-1}_{\bar 1}(q^ns_1z):$. 
\end{proof}

\begin{thm}\label{lem:TETF}
The following relations hold:
\begin{align}
&[T(z),E(w)]=-qa\delta(q^{n+1}s_1w/z):W'(q^{n+1}s_1w)\frac{\Lambda'(w)}{\Lambda'(q^2w)}E(q^2w):\,,
\label{TEcom}
\\
&[T(z),F(w)]=s_3a\delta(q^{-n-1}w/z):W'(q^{-n-1}w)\frac{\Lambda'(w)}{\Lambda'(q^{-2}w)}F(q^{-2}w):\,,
\label{TFcom}
\end{align}
where $a=(s_1-s_1^{-1})(s_3-s_3^{-1})/(q-q^{-1})$.
\end{thm}
\begin{proof}
We shall consider only \eqref{TEcom}, since \eqref{TFcom} is quite similar.

It follows from Lemma \ref{lem:TE} that terms in $T(z)E(W)$ which have poles 
 at $z=q^{n+1}s_1w$ are of the form 
$T_i(z)E_{1,\ldots,1,\nu_{i-1},\ldots,\nu_1}(w)$, $i=1,\ldots,n$, or $T_0(z)E_{1,\ldots,1}(w)$. 
On the other hand, from \eqref{xi-explicit} we obtain 
\begin{align*}
&\bs{i}_{q^{n+1}s_1}\cdot \tilde\xi_{1,\ldots,1,1,\nu_{i-1},\ldots,\nu_1}
=\tau_{q^2}\bigl(\tilde\xi_{1,\ldots,1,0,\nu_{i-1},\ldots,\nu_1}\bigr)\,,
\quad i=1,\ldots,n\,,
\\
&\bs{0}_{q^{n+1}s_1}\cdot \tilde\xi_{1,\ldots,1}
=\tau_{q^2}\bigl(\tilde\xi_{1,\ldots,1}\bigr)\,.
\end{align*}
Therefore the residue of $T(z)E(w)$ at $z=q^{n+1}s_1w$ consists of the same terms as the
normal-ordered expression in the right hand side of \eqref{TEcom}. Commutativity with screening operators then forces
that they are proportional.

It remains to show that the other poles 
$z=q^{-n+2j+1}s_1w$, $j=0,\ldots,n-1$, 
are absent in the product $T(z)E(w)$.
By induction on $n$, it is enough to check that the residues at $z=q^{n-1}s_1w$ pairwise cancel out. 
To see that, we use the relations
\begin{align*}
&(\bs{n-k})_{q^{n-k}s_1}\cdot \tilde\xi_{1^{(l)},0,\nu}
=(\bs{n-l})_{q^{n-1}s_1}\cdot \tilde\xi_{1^{(k)},0,1^{(l-k)},\nu}\,,
\end{align*}
for all $k,l$ with $0\le k<l\le n$ and $\nu\in \{0,1\}^{n-l-1}$, where $1^{(a)}=\overbrace{1,\ldots,1}^a$. 

\end{proof}

\medskip

\begin{rem}\label{remV}
The commutation relations derived in Theorems \ref{prop:EE}, \ref{[EF] lem} can be rewritten as follows.
Set $\alpha=n/(2n-(2n-1)\gamma)$ and introduce 
\begin{align*}
&f_{\xi,\xi}(z,w)=(s_1z^2)^{\alpha}\exp\Bigl(-\sum_{r>0}\frac{1}{r}\bigl(\frac{w}{z}\bigr)^r
\frac{q^{-r}s_1^{-r}-q^rs_1^r}{q^{2nr}s_1^r-q^{-2nr}s_1^{-r}}\frac{q^{2nr}-q^{-2nr}}{q^r-q{-r}}\Bigr)\,,
\\
&f_{\xi,\eta}(z,w)=(s_1z^2)^{-\alpha}\exp\Bigl(\sum_{r>0}\frac{1}{r}\bigl(\frac{w}{z}\bigr)^r
\frac{q^{-r}s_1^{-r}-q^rs_1^r}{q^{2nr}s_1^r-q^{-2nr}s_1^{-r}}\frac{q^{nr}-q^{-nr}}{q^r-q{-r}}(q^{nr}s_1^r+q^{-nr}s_1^{-r})\Bigr)\,.
\end{align*}
Then 
\begin{align*}
&f_{\xi,\xi}(z,w)V_\xi(z)V_{\xi}(w)=f_{\xi,\xi}(w,z)V_{\xi}(w)V_\xi(z)\,,\\
&f_{\xi,\eta}(z,w)V_\xi(z)V_{\eta}(w)-f_{\eta,\xi}(w,z)V_{\eta}(w)V_\xi(z)\\
&=\sum_{k=0}^ns_3^{-1}a_{n,k}\delta\bigl(a^{2n-2ks_1z/w}\bigr)V_{\chi_{k,1}}(z)
+\sum_{k=0}^ns_3a_{n,k}\delta\bigl(a^{2n-2ks_1w/z}\bigr)V_{\chi_{\bar k,1}}(w)\,.
\end{align*}
\end{rem}
\qed

\subsection{Extended algebra for $\osp(2|2n)$}\label{boson osp sec}

The case $\mathfrak{osp}(2|2n)$ is very similar to that of $\gl(2n|1)$. 
We give only the relevant formulas and state the results. 

The Cartan matrix is indexed by $I=\{n,\ldots,1,\bar1\}$.
The non-zero entries of  the matrix $K$ are 
\begin{itemize}
\item $K_{i,i}=-2\gamma$ for $i\neq 1,\bar1$ and  $K_{1,1}=K_{\bar1,\bar1}=-1$,
\item $K_{i,i\pm1}=\gamma$, $i\neq 1$, $K_{2,\bar1}=K_{\bar1,2}=\gamma$,
and $K_{1,\bar1}=K_{\bar1,1}=1-2\gamma$,
\end{itemize}
and $\det K=4(-\gamma)^{n}(\gamma-1)\neq 0$. 
For example, for $n=4$ we have 
\begin{align}\label{osp(2|8)K}
K_{\mathfrak{osp}(2|8)}=\begin{pmatrix}
-2\gamma & \gamma & 0 & 0 & 0  \\
\gamma & -2\gamma & \gamma & 0 & 0 \\
0 & \gamma & -2\gamma & \gamma & \gamma  \\
0 & 0 & \gamma &-1 & 1-2\gamma \\
0& 0 & \gamma & 1-2\gamma & -1 \\
\end{pmatrix}.
\end{align} 

The bosonization of the $qq$-character $\chi_{1,1}$ given in Lemma \ref{xi explicit lem}, reads
\begin{align*}
V_{\chi_{1,1}}(z)=\sum_{j=1}^n q^{2j-1}s_1 M_{i,1}(z)+\frac{s_1-s_1^{-1}}{q-q^{-1}}
\bigl(M_{0,1}(z)+M_{\bar{0},1}(z)\bigr)+\sum_{j=1}^n q^{-2j+1}s_1^{-1} M_{\bar j,1}(z)\,.
\end{align*}

The definition of $\Lambda(z)$, $\bar\Lambda(z)$  are
changed to  
\begin{align*}
\Lambda(z)=:\Lambda'(z)\frac{Y_1(z)}{Y_{\bar1}(s_1^2z)}:\,,\quad
\bar\Lambda(z)=:\bar\Lambda'(z)\frac{Y_{\bar1}(z)}{Y_1(s_1^2z)}:\,.
\end{align*}
The contractions \eqref{LaLa}, \eqref{LaS}
stay the same except 
\begin{align*}
\cont{\Lambda(z)}{S_{\bar{1}}(w)}=-\cont{S_{\bar{1}}(w)}{\Lambda(z)}=s_1^2z-w 
\end{align*}
and the same relation with $1\leftrightarrow \bar{1}$, 
$\Lambda(z)\leftrightarrow \bar{\Lambda}(z)$ interchanged.

The current 
$E(z)$ corresponding to $qq$-character $\xi_{1}$ given in Lemma \ref{xi osp explicit lem} is 
defined by the same formulas \eqref{Vxi}, \eqref{E(z)}. 
We also consider a rationalization \eqref{Tmodify} of $V_{\chi_{1,1}}(z)$.

The $EE$, $FF$ and $TE$, $TF$ relations are proved in the same way as in the case of $\gl(2n|1)$.
\begin{thm}\label{prop:EE osp}
The product $E(z)E(w)$ (resp.  $F(z)F(w)$)
is regular except for a simple pole at $z=s_3^2w$ (resp.  $z=s_3^{-2}w$). 
We have the quadratic relations
\begin{align}
&(z-s_3^2w)E(z)E(w)+(w-s_3^2z)E(w)E(z)=0\,,\label{EE-osp}\\
&(z-s_3^{-2}w)F(z)F(w)+(w-s_3^{-2}z)F(w)F(z)=0\,.\label{FF-osp}
\end{align}
\qed
\end{thm}

\medskip 

\begin{thm}\label{lem:TETF osp}
We have the relations
\begin{align*}
&[T(z),E(w)]=
-aq\delta(q^{n+1}s_1w/z):W'(q^{n+1}s_1w)\frac{\Lambda'(w)}{\Lambda'(q^2w)}E(q^2w):\\
&\quad +a
s_3\delta(q^{-n-1}s_1w/z):W'(q^{-n-1}s_1w)\frac{\Lambda'(w)}{\Lambda'(q^{-2}w)}E(q^{-2}w):
\,,\\
&[T(z),F(w)]=
-aq\delta(q^{n+1}s_1w/z):W'(q^{n+1}s_1w)\frac{\Lambda'(w)}{\Lambda'(q^2w)}F(q^2w):\\
&\quad
+as_3\delta(q^{-n-1}s_1w/z):W'(q^{-n-1}s_1w)\frac{\Lambda'(w)}{\Lambda'(q^{-2}w)}F(q^{-2}w):
\end{align*}
with $a=(s_1-s_1^{-1})(s_3-s_3^{-1})/(q-q^{-1})$.
\qed
\end{thm}

As we noted before, unlike the case of  $\gl(2n|1)$, 
the $qq$-characters $\chi_{k,1}$ with $k\ge2$ are not basic in general. 
This means that the bosonization $V_{\chi_{k,1}}(z)$ 
is not a sum of pure vertex operators but involves also derivatives of them. 
Nevertheless we expect to have the relation of the form
\begin{align*}
[E(z),F(w)]=\sum_{k=0}^{n-1}a_{n,k}\delta(q^{2n-2k} z/w)T_k(q^{n-k}s_1z)
+\sum_{k=0}^{n-1}a_{n,k}\delta(q^{2n-2k} w/z)\bar{T}_{k}(q^{n-k}s_1w)\,,
\end{align*}
where the currents $T_k(z),\bar{T}_k(z)$ are rationalizations of $V_{\chi_{k,1}}(z)$, $V_{\chi_{\bar k,1}}(z)$
and $a_{n,k}$ are some constants.

\subsection{Extended algebra for $\gl(n|m)$}\label{boson glnm sec}
Finally we give a few words about the case  
of $\gl(n|m)$ with standard parity 
discussed in Subsection \ref{gl(n|m) sec}. 

In the bosonization \eqref{Vchi} of a basic $qq$-character, 
the coefficients $c_m$ are to be determined from the 
commutativity with screening operators. More specifically we  
demand that two terms related by a root current pairwise cancel. 
In the present case of $\gl(n|m)$, this means the following.
Let $m_1,m_2$ be monomials which are related as
\begin{align*}
m_2=m_1\times
\begin{cases}
A_{i,aq}^{-1}& (i=1,\ldots,n-1),\\  
A_{i,as_1}^{-1}& (i=\bar1,\ldots,\overline{m-1}),\\ 
A_{0,as_3^{-1}}^{-1}& (i=0).\\  
\end{cases}
\end{align*}
Set $m_1=Y_{i,a}\prod_{b\neq a}Y_{i,b}^{n_{i,b}}M$,
$M=\prod_{j\neq i}\prod_bY_{j,b}^{n_{j,b}}$. 
Then the ratio $c_{m_2}/c_{m_1}$ is given by $c_{m_1,m_2}$ defined as
follows: 
\begin{align}
c_{m_1,m_2}=
\begin{cases}
q^{-2}\prod_{b\neq a} \omega_2(b/a)^{n_{i,b}}
&(i=1,\ldots,n-1),\\ 
s_1^{-2}\prod_{b\neq a} \omega_1(b/a)^{n_{i,b}}&
(i=\bar1,\ldots,\overline{m-1}),\\ 
-\prod_{b\neq a} \omega_0(b/a)^{n_{i,b}}&
(i=0),\\ 
\end{cases}
\label{twoterm}
\end{align}
where
\begin{align*}
\omega_2(x)=\frac{1-s_3^2x}{1-x}\frac{1-s_1^2x}{1-q^{-2}x}\,,
\quad
\omega_1(x)=\frac{1-s_3^2x}{1-x}\frac{1-q^2x}{1-s_1^{-2}x}\,,
\quad
\omega_0(x)=s_3^{-2}\frac{1-s_3^2x}{1-x}\,.
\end{align*}
Starting from the top monomial and applying the rule \eqref{twoterm}
we obtain  the following bosonization of $\chi_{1,1}$:
\begin{align*}
V_{\chi_{1,1}}(z)=
(s_2-s_2^{-1})\sum_{i=1}^{n}q^{2i-1}M_{i,1}(z)
+(s_1-s_1^{-1})\sum_{i=1}^{m}s_1^{-2i+1}M_{\bar1,1}(z)
\,,
\end{align*}
where $M_{i,a}$ as before
are monomials in $\chi_{1,1}$, see \eqref{chi-glnm}.
(They are not to be confused with those in Sections \ref{gl(2n|1) sec} and 
\ref{boson gl(2n|1) sec}.)

Consider now the bosonization of the $qq$-character $\xi_1$. 
It is written as a sum over semi-standard Young tableaux
 $T$ of shape $\lambda^{(0)}=(m,\ldots,m)$,
\begin{align}
&\xi_1=\sum_T \tilde{\xi}_T\,,\quad 
\tilde{\xi}_T=
\tau_{q^{n+1}s_1^{2m+1}}\Bigl(\bs{0}_{q^{-n-1}s_1^{-1}}
\prod_{i=1}^n\prod_{j=1}^m M_{T(i,j),
q^{-2i}s_1^{-2j}}\Bigr)\,.
\label{xi-tab}
\end{align}
The coefficient of a given monomial 
can be determined by applying \eqref{twoterm} repeatedly. 
The following lemma ensures that the result does not depend on the 
way of computation. 

\begin{lem}
Let 
$T_1,T_2,T_3,T_4$ be four semi-standard tableaux 
which differ from each other by only 
two places $(k,l),(k',l')$,
\begin{align*}
&\bigl(T_1(k,l),T_1(k',l')\bigr)=(r,s),\quad 
\bigl(T_2(k,l),T_2(k',l')\bigr)=(r,s'),\\
&\bigl(T_3(k,l),T_3(k',l')\bigr)=(r',s),\quad 
\bigl(T_4(k,l),T_4(k',l')\bigr)=(r',s'),
\end{align*}
such that the corresponding factors in \eqref{xi-tab}
are related as
$M_{r',a'}=M_{r,a}A_{j,a''}^{-1}$ and
$M_{s',b'}=M_{s,b}A_{j',b''}^{-1}$ with some 
$j,j'$ and $a,a',a'',b,b',b''$. Then we have
\begin{align}
c_{T_1,T_2}c_{T_2,T_4}=c_{T_1,T_3}c_{T_3,T_4}\,,
\label{cccc}
\end{align}
where $c_{T,T'}=c_{\tilde{\xi}_T,\tilde{\xi}_{T'}}$.
\end{lem}
\begin{proof}
There are the following non-trivial cases to consider:
\begin{align*}
(1)\quad &M_{r,a}=M_{i+1,a},\ M_{r',a'}=M_{i,aq^2} \,,
\quad M_{s,b}=M_{i+1,b},\ M_{s',b'}=M_{i,bq^2}\,
\quad (i=1,\ldots,n-1)\,,\\
(2)\quad &M_{r,a}=M_{1,a},\ M_{r',b'}=
M_{\bar{1},a s_3^{-2}} \,,
\quad M_{s,b}=M_{1,b},\ M_{s',b'}=M_{\bar{1},bs_3^{-2}}\,,\\
(3)\quad &M_{r,a}=M_{\bar{i},a}, \, 
M_{r',a'}=M_{\overline{i+1},as_1^2} \,,
\quad M_{s,b}=M_{\bar{i},b},\ 
M_{s',b'}=M_{\overline{i+1},bs_1^2}\,
\quad (i=1,\ldots, m-1)\,,\\
(4)\quad &M_{r,a}=M_{i,a},\ M_{r',a'}=M_{i-1,aq^2}\,,\quad 
M_{s,b}=M_{i+1,b},\ M_{s',b'}=
M_{i,bq^2}\,\quad (i=2,\ldots, n-1)\,,\\
(5)\quad &M_{r,a}=M_{1,a},\ 
M_{r',a'}=M_{\bar{1},as_3^{-2}} \,,\quad 
M_{s_b}=M_{2,b},\ M_{s',b'}=M_{1,bq^2}\,,\\
(6)\quad &M_{r,a}=M_{\bar{1},a},\ 
M_{r',a'}=M_{\bar{2},as_1^2} \,,\quad 
M_{s,b}=M_{1,b},\ M_{s',b'}=M_{\bar1,bs_3^{-2}}\,,\\
(7)\quad &M_{r,a}=M_{\bar{i},a}, \, 
M_{r',a'}=M_{\overline{i+1},as_1^2} \,,\quad 
M_{s,b}=M_{\overline{i+1},b},\ 
M_{s',b'}=M_{\overline{i+2},bs_1^2}\,
\quad (i=2,\ldots,m-1)\,.
\end{align*}
In the case (1) we compute
\begin{align*}
&\frac{c_{T_3,T_4}}{c_{T_1,T_2}}=
\frac{\omega_2(s_2^2 b/a)^{-1}}{\omega_2(b/a)}\,, 
\quad \frac{c_{T_2,T_4}}{c_{T_1,T_3}}= 
\frac{\omega_2(s_2^2 a/b)^{-1}}{\omega_2(a/b)}\,,
\end{align*}
from which \eqref{cccc} follows in view of the relation
\begin{align}
&\omega_2(x)=\omega_2(s_2^2/x)=\frac{\omega_0(x)}{\omega_0(s_1^2x)}\,.
\label{omega-id}
\end{align}

The other cases can be checked similarly, 
using \eqref{omega-id} and changing the roles of $s_1$ and $s_2$ as necessary.
\end{proof}

The bosonization of  $\eta_1=\bs{0}^1$ is trivial. 
Thus we obtain an extension of the $W$ algebra of type $\gl(n|m)$. 
In this paper we do not discuss the relations of these currents. 

\bigskip

\begin{rem}\label{remMac}
It would be interesting to describe the coefficients $c_m$ 
of a bosonization of $qq$-characters in a general situation. 
In this regard, we note 
the following intriguing result due to \cite{FHSSY}.

Consider $\gl(n)=\gl(n|0)$, and 
let $V_{\chi_1}(z)=\sum_{i=1}^n u_i\Lambda_i(z)$ 
be the bosonization of the basic $qq$-character $\chi_1$ (i.e. 
the fundamental current of the $W(\gl_n)$ algebra).
Here $\Lambda_i(z)$'s are vertex operators and $u_i$'s are evaluation parameters,
see e.g. \cite{FJMV}, Section 3.1.
Let $\chi_\lambda$ be the current corresponding to a partition $\lambda$. 
Then we have (choosing an overall multiple appropriately) 
\begin{align*}
V_{\chi_\lambda}(z)=\sum_{T}\psi_T u^T :\Lambda_T(z): \,,
\end{align*}
where $u^T=\prod_{i=1}^{\ell(\lambda)}\prod_{j=1}^{\lambda_i}u_{T(i,j)}$, 
the sum is taken over semi-standard tableaux $T$ 
of shape $\lambda$, 
and $\psi_T$'s denote the Pieri coefficients for the Macdonald polynomials
\footnote{Macdonald's  $q$ and $t$ are $s_1^2$ and $s_2^{-2}$ in the current 
notation.}
(see (7.11') and (7.13') in \cite{M}).
In particular, the vacuum expectation value of $V_{\chi_\lambda}(z)$ 
coincides with the Macdonald polynomial
\begin{align*}
P_\lambda=\sum_{T}\psi_T u^T=\langle V_{\chi_\lambda}(z)  \rangle\,.
\end{align*}
\qed
\end{rem}
\bigskip 

{\bf Acknowledgments.\ }
The study of BF has been funded within the framework of the HSE University Basic Research Program. MJ is partially supported by 
JSPS KAKENHI Grant Number JP19K03549. 
EM is partially supported by Simons Foundation grant number \#709444.

BF thanks Jerusalem University for hospitality. EM thanks Rikkyo University, where a part of this work was done, for hospitality.


\begin{thebibliography}{0000000}

\bibitem[B]{B} M. Bershadsky, {\it Conformal field theories via Hamiltonian reduction}, Commun. Math. Phys. {\bf 139} (1991), 71

\bibitem[BP]{BP} P. Bouwknegt and K. Pilch, 
{\it On deformed $\mathcal{W}$-algebras and quantum affine algebras}, 
{\it Adv. Theor. Math. Phys.} {\bf 2} (1998) 357--397

\bibitem[FHSSY]{FHSSY} B. Feigin, A. Hoshino, J. Shibahara, J. Shiraishi,
and S. Yanagida, 
{\it Kernel function and quantum algebras},
RIMS k{\= o}ky{\= u}roku {\bf 1689} (2010), 133–-152

\bibitem[FJMV]{FJMV} B. Feigin, M. Jimbo, E. Mukhin, and I. Vilkoviskiy, {\it Deformations of $\mc W$  algebras via quantum toroidal algebras}, {\it Selecta Math. } {\bf 27} (2021), no. 52,  62 pp

\bibitem[FJM]{FJM} B. Feigin, M. Jimbo, and E. Mukhin,
{\it Combinatorics of vertex operators and deformed W-algebra of type D$(2,1;\alpha)$}, Adv. Math. 403 (2022), Paper No. 108331, 54 pp

\bibitem[FM]{FM} E. Frenkel and E. Mukhin, 
{\it Combinatorics of $q$-characters of finite-dimensional representations of quantum affine algebras}, Comm. Math. Phys., {\bf 216} (2001), no. 1, 23–-57

\bibitem[FR]{FR} E. Frenkel and N. Reshetikhin, {\it Deformations of W-algebras associated to simple Lie algebras}, Comm. Math. Phys. {\bf 197} (1998), no. 1, 1–-32

\bibitem[FS]{FS} 
B. Feigin and A. Semikhatov, {\it $W^{(2)}_n$ algebras}, Nucl.Phys. {\bf  B698} (2004), 409--449


\bibitem[H]{H} K. Harada, {\it Quantum deformation of Feigin-Semikhatov's W-algebras and 5d AGT correspondence with a simple surface operator},  
arXiv:2005.14174, 1--44


\bibitem[KP1]{KP1} T. Kimura and V. Pestun, 
{\it Quiver W-algebras}, Lett. Math. Phys. 108 (2018) 1351–1381 
 
\bibitem[KP2]{KP2}
T. Kimura and V. Pestun, 
{\it Fractional quiver W-algebras}, 
Lett. Math. Phys. {\bf 108} (2018), no. 11, 2425--2451

\bibitem[K1]{K1} T. Kojima, {\it Quadratic relations of the deformed W-superalgebra $W_{q,t}(A(M,N))$}, J. Phys. A, {\bf{54}} (2021), no. 33, Paper No. 335201, 37 pp

\bibitem[K2]{K2} T. Kojima, {\it Quadratic Relations of the Deformed $W$-Algebra for the Twisted Affine Lie Algebra of Type A$^(2)_{2N}$},  arXiv:2108.13883, 1--36

\bibitem[LM]{LM} K. Lu and E. Mukhin, {\it Jacobi-Trudi identity and Drinfeld functor for super Yangian}, International Mathematics Research Notices 2021(21)
DOI:10.1093/imrn/rnab023

\bibitem[M]{M} I. Macdonald, {Symmetric Functions and Hall Polynomials},
2nd ed., Oxford Mathematical Monographs,
Clarendon Press, 1995.

\bibitem[N]{N} N. Nekrasov, 
{\it BPS/CFT correspondence: non-perturbative Dyson-Schwinger equations and $qq$-characters},
JHEP {\bf 1603} (2016), 181

\bibitem[P]{P} A. Polyakov, {\it Gauge transformations and diffeomorphisms}, Int. J. Mod. Phys. A5 (1990), 833 



\end{thebibliography}
\end{document}